\documentclass[12pt]{amsart}

\usepackage{amssymb,amsfonts}
\usepackage[T1]{fontenc}
\usepackage{amsthm}
\usepackage{graphicx}
\usepackage[all,arc]{xy}
\usepackage[colorlinks=true, allcolors=blue]{hyperref}
\usepackage{enumerate}
\usepackage{bbm}
\usepackage{dsfont}
\usepackage{mathrsfs}
\usepackage{tikz-cd}
\usepackage{import}
\usepackage[utf8]{inputenc}
\usepackage{hyperref}
\usepackage{tikz}
\usepackage{multirow}
\usepackage{booktabs} 
\usepackage{array} 

\usepackage[left=1in,top=1in,right=1in,bottom=1in]{geometry}

\usepackage{mathtools}
\usepackage{subcaption}

\newtheorem{thm}{Theorem}[section]
\newtheorem{cor}[thm]{Corollary}
\newtheorem{prop}[thm]{Proposition}
\newtheorem{remark}[thm]{Remark}
\newtheorem{lem}[thm]{Lemma}
\newtheorem{conjecture}[thm]{Conjecture}
\newtheorem{quest}[thm]{Question}

\theoremstyle{definition}
\newtheorem{defn}[thm]{Definition}

\newtheorem{example}[thm]{Example}
\newtheorem{rmk}[thm]{Remark}

\newcommand{\Q}{\mathbb{Q}}

\newcommand{\coker}{\operatorname{coker}}
\newcommand{\conf}{\operatorname{Conf}}
\newcommand{\uconf}{\operatorname{UConf}}
\newcommand{\im}{\operatorname{im}}
\newcommand{\ind}{\operatorname{Ind}}
\newcommand{\triv}{\mathrm{triv}}
\newcommand{\reg}{\mathrm{reg}}
\newcommand{\Sp}{\mathrm{Sp}}

\newcommand{\cy}[1]{\textcolor{red}{#1}}

\title{Computing stable homology representations of graph configuration spaces}
\author{Eric Ramos}
\address{Department of Mathematical Sciences, Stevens Institute of Technology, Hoboken, NJ, United States}
\email{eramos3@stevens.edu}

\author{Claudia He Yun}
\address{Department of Mathematics and Statistics, UiT The Arctic University of Norway, Troms{\o}, Norway}
\email{he.yun@uit.no}

\begin{document}

\begin{abstract}
Configuration spaces of graphs frequently grow factorially in complexity with the number of particles they parametrize. However, for suitable families of nested graphs $G_\bullet$ with compatible symmetric group actions, Ramos and White prove that, for fixed $k$, the rational homology of the $k$\textsuperscript{th} configuration spaces of $G_\bullet$ has multiplicity stability. In the current work, we derive the stable range and use computer algebra to determine the stable representations on homology for $k=2$ and $G_\bullet$ several families of graphs, including the complete graphs, the complete bipartite graphs on $2n$ vertices, the crown graphs on $2n$ vertices, and the complete tripartite graphs on $2n+1$ vertices. We determine the stable multiplicities for certain irreducible components in the case $k=3$ and $G_\bullet$ the complete graphs.
\end{abstract}
	
\maketitle

\section{Introduction}

Let $X$ denote a topological space, and $n \geq 1$ an integer. The \textbf{configuration space of $X$ on $n$ particles} is defined to be the space
\[
\conf_n(X) := \{(x_1,\ldots,x_n) \mid x_i \neq x_j \text{ if } i \neq j\}.
\]

Configuration spaces of manifolds and algebraic varieties have been a mainstay of algebraic topology for around a century dating back to work of Arnol'd \cite{Arn}. Notably, it has only been in the last twenty years or so that researchers have begun to consider configuration spaces whose underlying spaces are more combinatorial in nature. Specifically, contemporaneous work of Abrams \cite{Abr}, Ghrist \cite{Ghr}, and {\'S}wi{\k{a}}tkowski \cite{Swia} at the turn of the century began the study of configuration spaces of graphs.

While one might expect the low dimensionality of graphs to make the resulting configuration spaces easier to understand, it was quickly discovered that there are a number of theoretical reasons why this is not the case.
Although graph configuration spaces are always homotopy equivalent to finite cubical complexes with fixed (i.e. not varying in $n$) dimensions
for $n \gg 0$ \cite{Abr,abrams-ghrist,Swia}, their Euler characteristics grow factorially with $n$ \cite{Gal}. That implies that the Betti numbers of these spaces must grow in a factorial fashion.
This factorial growth 
can be seen in the $n$-particle configuration space of the line segment, which has $n!$ connected components. In particular, a number of well known stability phenomena for configuration spaces with growing numbers of particles cannot possibly hold in this setting outside of the most simple possible graphs such as star trees \cite{wawrykow2025homology}.

One consequence of the combinatorial explosion in homology is that there are famously very few explicit computations of the homology groups of graph configuration spaces. The most obvious way to mitigate this factorial growth is to quotient by the natural symmetric group action, which creates the \textbf{unordered configuration spaces}. Although computations are more common for these quotient spaces, the situation is still fairly sparse. For instance, while we currently have a decent grasp of what $H_1$ looks like for any graph \cite{KP}, what $H_2$ looks like for unordered configurations of planar graphs \cite{AK}, what $H_i$ looks like for trees \cite{Farl}, and the Betti numbers for a large (but finite) collection of small sporadic graphs \cite{Drum}, there is still a lot we do not understand. The question of whether these homology groups can ever admit odd torsion is still very open, and has implications to theoretical physics \cite{MS}. Section \ref{section:confbackground} contains a more detailed list of known results.

The above being said, however, it was noted in works of L\"utgehetmann \cite{Lut}, as well as the first author and White \cite{RW} that there were observable stable behaviors if one \emph{fixed} the number of points being configured. Define FI to be the category whose objects are finite sets and whose morphisms are injections. An \textbf{FI-graph} is defined to be a functor $G_\bullet$ from FI to the category of graphs and graph homomorphisms. More concretely, an FI-graph can be thought of as a collection of graphs $\{G_n\}_{n \in \mathbb{N}}$ such that for any injection of sets $f:[a] \rightarrow [b]$ there is a naturally induced homomorphism of the associated graphs $f_\ast:G_a \rightarrow G_b$.

The first example of an FI-graph is the collection of complete graphs on $n$ vertices. 
Other examples include complete bipartite graphs $K_{n,n}$, crown graphs $W_n$, i.e. $K_{n,n}$ with a perfect matching removed, and Kneser (resp. Johnson) graphs of disjoint (resp. maximally) overlapping subsets of a fixed size. Importantly, as all injections from a finite set to itself are necessarily permutations, the above definition of an FI-graph necessitates that each graph $G_n$ is acted on by the corresponding symmetric group $S_n$ (by graph homomorphisms), and that all of these actions are in a way compatible with each other due to the variety of maps being induced from the other injections. It is this ``abundance of symmetry" that allows one to conclude a number of impressive stable behaviors that manifest in various quantities related to an FI-graph. These range from standard counting invariants such as counting subgraphs \cite{RW}, to extremal invariants such as independence numbers \cite{GR}, to probabilistic quantities such as hitting numbers of random walks \cite{RW2}.

Our current work relies on the following theorem.
Note that in the statement of this theorem we will write $\Sp(\lambda)$ to denote the irreducible representation of the symmetric group $S_n$ associated to the partition $\lambda$ of $n$. If $\mu = (\mu_1,\ldots,\mu_r)$ is a partition of some other integer $m < n - \mu_1$, We also write $\mu[n]$ to denote the partition of $n$ given by $\mu[n] = (n-m,\mu_1,\mu_2,\ldots,\mu_r)$. Using Young diagrams, $\mu[n]$ is the partition obtained from $\mu$ by adding a new long first row, whose length is chosen to make the result a partition of $n$.

\begin{thm}[Ramos and White \cite{RW}, Theorem G]\label{thm: mainTech}
    Let $G_\bullet$ denote a finitely generated FI-graph (see Definition \ref{def: FIgraph}) Then for any fixed $i,m \geq 0$ there exists a finite set of partitions $\Lambda$ as well as a finite set of integers $\{m_\lambda\}_{\lambda \in \Lambda}$ such that for every $n \gg 0$,
    \[
    H_i(\conf_m(G_n);\Q) \cong \bigoplus_{\lambda \in \Lambda}m_\lambda \Sp(\lambda[n]) \label{eq: stable}
    \]
\end{thm}

The above theorem therefore provides a natural paradigm for computing the homology groups of infinitely many configuration spaces by performing a finite computation. One simply has to use theoretical tools to determine when the isomorphism (\ref{eq: stable}) begins to hold. Such tools in the most general contexts have been proven through the FI literature (see \cite{CEF,LR,CMNR,CMNR}). In this work we specialize these general theorems to our very particular context in order to perform explicit computations. Importantly, because the above theorem is equivariant in nature, our computations reveal more than just Betti number counts, but also the deeper representation theoretic structure accounting for the action of the automorphism groups of the underlying graph.

The primary goals of this work can therefore be summarized as follows:

\begin{enumerate}
    \item To provide explicit computations of the rational homology groups of configuration spaces with fixed numbers of points in a number of infinite families of graphs including star graphs (Theorem \ref{thm: star}), complete graphs (Theorem \ref{thm: ordered complete} and Theorem \ref{thm: three particles on complete}), complete bipartite graphs (Theorem \ref{thm: ordered bipartite}), crown graphs (Theorem \ref{thm: ordered crown}), and complete tripartite graphs (Theorem \ref{thm: ordered tripartite}). Importantly, these computations account for the representation theoretic structure coming from the automorphism groups of these graph families.
    \item Illustrate a proof of concept for the usage of representation stability and the theory of FI-modules for explicit computations in a vein similar to \cite{MMPR}.
    \item Provide a number of conjectures and possible directions for future research through observations in our experimental data.
\end{enumerate}

\begin{remark}
    Computations of the homology groups of configurations of two points on the complete graph have appeared already in various forms in the literature. For instance, \cite{BF,FH} were the first to compute the Betti numbers in this case, while the much more recent \cite{GG} was able to deduce the cohomology ring structure. Our work is the first to fully describe the underlying representation theoretic structure in this case. We also cover more families of graphs than those prior works.

    More recently, computations for the homology groups of configurations of star graphs have been accomplished by Wawrykow \cite{wawrykow2025homology}
\end{remark}

In the next section we will discuss a number of possible future directions that naturally arise from this work. In that discussion we will explicitly point out a number of experimental observations that specifically rely on our access to the more finely grained equivariant description of the homology groups.

\subsection{Future directions}

As discussed in the prior section, we view this paper as more of a first step, or proof of concept, for what is possibly a larger research program in computational topology. In this section we discuss a number of future directions that we believe are interesting following the results of this work.

\subsubsection{Expanding topological reductions}

The majority of computations performed in this work are focused on the two particle case. The reason for this is two fold: By applying the Abrams model for configuration space \cite{Abr}, the number of cells in each dimension grows factorially with the number of particles and our derived bounds on when stable behaviors begin grow polynomially in the homological index and the number of particles. We do this mainly because the Abrams model does have a significant benefit of being by far the easiest to implement in actual coding.

Using the techniques of discrete Morse theory, Gonz\'ales and Gonz\'ales \cite{GG} construct what is an extremely computationally efficient cellular module for configurations of two points on complete graphs. This model was not directly applicable for our intended use as their vector field is not equivariant in the symmetric group action on $K_n$. However, the creation of symmetric discrete gradient vector fields has recently been a topic of some interest \cite{Y}.

We believe that it is very likely that by using slightly more sophisticated topological reductions and computational topology algorithms that the tools of this paper can be used to expand computations to many more particles. This would also be helpful in terms of expanding our results to work over more general rings.

\subsubsection{Machine learning assistance}

Arguably, the most significant boon to experimental mathematics has been the advent of powerful machine learning methods \cite{AIDiscover,funsearch,funsearch2,CEWW}. It is our belief that computations in the homologies of graph configuration spaces is a natural candidate for these transformative methods. For instance, can one use one of the aformeentioned architectures to find a graph whose configuration spaces have interesting torsion?

\subsubsection{Graph properties manifesting in homology}

One of the earliest great triumphs of graph configuration space theory is due to Ko and Park \cite{KP}. In that work, they prove that $H_1$ of \emph{unordered} configurations of a graph $G$ (i.e. the quotient of the usual configuration space by the permutation action on points) will contain torsion if and only if $G$ is non-planar. This of course opened the door to a large number of questions of the form ``what graph theoretic properties can the configuration space detect?" Since the work of Ko and Park, this question has primarily been answered in terms of finding various connectivity invariants manifesting in the homology groups of configuration spaces \cite{R2,ADCK2}. In particular, there have yet to be any further results that point to significant \emph{topological} invariants of the graph being detected in configuration space.

Through the computations of this work, we have come to the following question:

\begin{quest}
Does the second homology group $H_2(\conf_n(G))$ detect whether $G$ has any \textbf{intrinsic linkage} for $n \gg 0$? 
\end{quest}

Looking at the tables \ref{table: ordered 2 pts Kn},\ref{table: H2Knn as SnxS2-reps}, and \ref{table: H2Crown with SnxS2}, for complete graphs, complete bipartite graphs, and crown graphs, one might observe a pattern in $H_2$. Firstly, in all three cases there is generally a significant increase in total dimension when the family transitions from being planar to being non-planar. This will be at $K_5, K_{3,3}$ and $W_{5}$, respectively. This behavior is, to a certain extent, expected due to the work \cite{AK}. However, what came as a surprise to us is the behavior of $H_2$ after these families transition from being linklessly embeddable to intrinsically linked; at $K_6, K_{4,4},$ and $W_6$, respectively. In all cases there was not only a meaningful increase in the dimension of the homology group, but also in the diversity of irreducible representations that appear. In other words, the dimensional increase is due to more than just the same irreducible representations gaining multiplicity before their stable range, it is also due to a large number of never before seen irreducible representations manifesting at precisely the moment when the family becomes intrinsically linked. We also included some computations for the Kneser graphs $K(n,2)$, see Table \ref{table: H2Kneser}, and the complete tripartite graphs $K_{n,n,1}$, see Table \ref{table: H2tripartite}, as these families contain members of the Petersen family, which forms the minimal forbidden minors for intrinsic linkage \cite{robertson1993linkless}.  

One possible explanation for this behavior is that the presence of these intrinsically linked cycles is creating new non-trivial relations in $\pi_1$. We believe in any case that there is some hope to think that this second homology group can detect intrinsic linkage.

\subsubsection{Higher dimensional complexes}

While we only look at graph configuration spaces in this work, the technical core comes from Theorem \ref{thm: mainTech}, which is ultimately much more general than we need. In particular, \cite{RW} proves a number of similar topological stability statements for other families of combinatorially motivated spaces. It is possibly interesting to apply the paradigm discussed here to things such as clique or independence complexes of FI-graphs, or even configuration spaces thereof.

All code used in this paper can be found at \texttt{https://github.com/ClaudiaHeYun/GraphConf/}.

\section*{Acknowledgments}
The first author was supported by NSF grant DMS-2452031. We would like to thank David Speyer and Jenny Wilson for illuminating discussions and Tobias Boege for computer help.

\section{The technical necessities}
\subsection{Representation Stability and FI-modules}\label{section:repstab}

At its core, the prevailing philosophy of the present work is that under certain ideal circumstances the computation of all homology groups within an infinite family can be reduced to a computation of only finitely many of them. To describe what we mean in this paper by ``ideal circumstances," we will need to use the language of representation stability and FI-modules.

\begin{defn}
    Let $\lambda = (\lambda_1,\ldots,\lambda_r)$ denote a partition of some non-negative integer $m$. Then for any $n \geq |\lambda| + \lambda_1$, we define the \textbf{padded partition} by $\lambda[n] = (n-|\lambda|,\lambda_1,\ldots,\lambda_r)$. Pictorially, $\lambda[n]$ is the partition of $n$ obtained from $\lambda$ by adding a sufficiently long first row.
    
    Given any partition $\lambda$, there is a naturally associated irreducible representation of the symmetric group $S_m$, called the \textbf{Specht module}, which we will denote $\Sp(\lambda)$. In the case of a padded partition $\lambda[n]$, we will use the abbreviated notation $\Sp(\lambda[n]) = \Sp_n(\lambda)$.

    Let $\{V_n\}_{n \geq 0}$ be a collection of $\Q$-vector spaces, where $V_n$ carries a linear action of the symmetric group $S_n$ for each $n \geq 0$. Then we say that the collection has \textbf{multiplicity stability} if there is a \emph{finite} collection of partitions $\Lambda$, along with integers $m_\lambda \geq 0$ for each $\lambda \in \Lambda$, such that for all $n \gg 0$
    \[
    V_n \cong \bigoplus_{\lambda \in \Lambda} m_\lambda \cdot \Sp_n(\lambda)
    \]
\end{defn}

\begin{example}
    Multiplicity stability can be thought of as asserting that for all $n$ sufficiently large, the representation $V_{n+1}$ is the ``same" as the representation $V_n$, even through these are representations of different groups. One very simple but illustrative example is the permutation representation $\Q^n$. This representation decomposes into a copy of the trivial representation and a copy of the standard representation. In other words, for all $n \geq 2$,
    \[
    \Q^n \cong \Sp_n(\emptyset) \oplus \Sp_n( (1) )
    \]
\end{example}

In modern terms, the most common way that one encounters multiplicity stability is through the theory of FI-modules.

\begin{defn}
    We use FI to denote the category whose objects are finite sets, and whose morphisms are injections. An FI-module is a (covariant) functor $V$ from FI to the category of finite dimensional $\Q$-vector spaces. In more concrete terms, an FI-module is a collection of finite dimensional $\Q$-vector spaces $\{V_A\}_{A \text{ finite}}$, one for each finite set $A$, such that for any injection $A \hookrightarrow B$ there is a naturally associated linear map $V_A \rightarrow V_B$.

    We notice that within FI, one has the full subcategory whose objects are the sets $[n] = \{1,\ldots,n\}$. In fact, FI is equivalent to this subcategory. It follows from this that it suffices to define FI-modules only on the sets $[n]$. Going forward we almost exclusively think about FI-modules in this way, with a handful of exceptions (see, for instance, Remark \ref{sharpallsets}). If $V$ is an FI-module, we denote $V([n])$ by $V_n$ for simplicity.
    
    Observe that, because injections from a finite set to itself are permutations, it is implicit in this definition that $V_n$ is an $S_n$-representation for each $n$.

    We say that an FI-module $V$ is \textbf{finitely generated} if there is an integer $N$ such that for all $n \geq N$, the vector space $V_{n+1}$ is spanned by the images $\iota(V_n)$ for all injections $\iota:[n] \hookrightarrow [n+1]$. In this case we say that $V$ has \textbf{generating degree $\leq N$}
\end{defn}

\begin{remark}
    There is nothing special about FI-modules being valued in $\Q$-vector spaces. Indeed, it is more common in recent times to consider FI-modules defined over arbitrary Noetherian rings. All of the results of this background section will continue to hold in this setting, of course excluding any that specifically talk about complex symmetric group representations. Because our eventual application will all be in the characteristic 0 case, we make the decision to only think about these types of modules in this paper.
\end{remark}

One may form the category of FI-modules by taking natural transformations as arrows. It can be seen that this category is abelian, as we can define things such as direct sums, kernels, and cokernels in a point-wise fashion. In particular, it is sensible to discuss the \textbf{submodules} and \textbf{quotient modules} of a given FI-module.

The connection between FI-modules and the aforementioned concept of multiplicity stability is summarized by the following theorem.

\begin{thm}\cite{CEF} \label{repstabthm}
Let $V$ be a finitely generated FI-module. Then the collection of symmetric group representations $\{V_n\}_{n \geq 0}$ is multiplicity stable.
\end{thm}

In order to apply Theorem \ref{repstabthm}, there are two very practical concerns that we must contend with. Suppose $V$ is an FI-module. Firstly, what tools does one have to prove that $V$ is finitely generated, and secondly, for which $n$ does $V_n$ have multiplicity stability, if $V$ is finitely generated?
To answer the first question, we have the following seminal theorem. Note that if $V$ is an FI-module then a \textbf{submodule} of $V$ is an FI-module $W$ such that $W_n \subseteq V_n$ for each $n$

\begin{thm} \cite{CEF} \label{Noetherianity}
    If $V$ is a finitely generated FI-module, then all submodules and quotient modules of $V$ must also be finitely generated.
\end{thm}

This theorem is particular useful because in topological contexts, the relevant FI-module often arises as the converging term of a spectral sequence, whose $E^2$-page is comprised of obviously finitely generated FI-modules.

Our second practical concern, on computing when exactly the multiplicities stabilize, is a considerably thornier issue. In what follows, we will outline the usual steps one takes to bound this stable range. We describe the method in its greatest generality, although in specific cases the bounds can often be significantly improved.

\begin{defn}
    Let $W$ be a fixed (finite dimensional) representation of some symmetric group $S_m$. Then the \textbf{free FI-module on $W$} is the FI-module $M(W)$ defined by the following assignment,
    \[
    M(W)_n := W \otimes_{S_m} \Q[\text{Hom}([m],[n])],
    \]
    where $\Q[\text{Hom}([m],[n])]$ is the vector space whose basis is indexed by injections from $[m]$ to $[n]$. More generally, we say that an FI-module is \textbf{free} if it is isomorphic to a direct sum of modules of the form $M(W)$. In the case where $W$ is the regular representation of $S_m$, we write $M(W) = M(m)$.
\end{defn}

Free FI-modules have played an extremely prominent role in the homological algebra of FI-modules since the origins of the subject (see, for instance, \cite{LR,R,CE,CMNR,N} for discussions of a number of important functors on the category of FI-modules for which free modules are acyclic). We observe the following facts about the free FI-module $M(W)$, where $W$ is an $S_m$-representation:

\begin{itemize}
    \item The generating degree of $M(W)$ is exactly equal to $m$;
    \item the multiplicities of irreducible representations in $M(W)_n$ stabilize no later than $n = 2m$.
\end{itemize}
The first of these facts follows immediately from the relevant definitions. The second fact is a bit more subtle, but ultimately follows from the fact that $M(W)_n \cong \text{Ind}_{S_m \times S_{n-m}}^{S_n} (W \otimes \Q)$, along with the Pieri rule.

It is the case that the free modules are precisely the projectives in the homological sense, at least over a field of characteristic 0. In fact, these modules will also be injective \cite{SS}. Because of these facts, identifying free modules in theoretical contexts often comes down to proving that certain derived functors vanish. In more practical contexts, however, the main tool one has for proving a given FI-module is free is via the following technique due to Church, Ellenberg, and Farb \cite{CEF}.

\begin{defn}
    Let FI$\sharp$ denote the category whose objects are sets of the form $[n] \sqcup \{\infty\}$ and whose morphisms are functions $f:[n] \sqcup \{\infty\} \rightarrow [m] \sqcup \{\infty\}$ satisfying the following:
    \begin{itemize}
        \item $f(\infty) = \infty$;
        \item $f|_{f^{-1}([m])}$ is an injection.
    \end{itemize}
\end{defn}

In the work \cite{CEF}, Church, Ellenberg, and Farb provide a different (but equivalent) definition of the category FI$\sharp$. Those authors thought of the category FI$\sharp$ as consisting of finite sets and injections that are only defined on a subset of their domain. It is clear that the definition given above is equivalent, and we prefer it as it makes it much more obvious how to define composition of morphisms. 
Going forward, when talking about a morphism, $f$, of FI$\sharp$, we will call the set of numbers not sent to $\infty$ the \textbf{true domain} of $f$.

\begin{rmk}\label{sharpallsets}
  The above paragraph illustrates why, when working with FI$\sharp$-modules, it is often more convenient to think of FI as the category of all finite sets. This way, every morphism in FI$\sharp$ can be thought of as a morphism in FI, albeit on a different domain.
\end{rmk}

We observe that FI can be naturally found as a subcategory of FI$\sharp$ by taking all maps that are injective on their entire domain. In particular, if $V$ is a module over FI$\sharp$, it can also be considered a module over FI by restriction. This leads us to the following theorem

\begin{thm}\cite[Theorem 2.24]{CEF}\label{sharpisfree}
Let $V$ denote a finitely generated module over the category FI$\sharp$. Then its restriction to FI is a free module.
\end{thm}

One of the most important results in the field is that, in some sense, all FI-modules are ``eventually" free.

\begin{defn}
    Let $\iota:FI \rightarrow FI$ denote the functor which sends $[n]$ to $[n+1]$, while sending an injection $[n] \hookrightarrow [m]$ to the injection $[n+1] \rightarrow [m+1]$ which extends it by sending $n+1$ to $m+1$. Then the \textbf{shift functor} for FI-modules is defined by
    \[
    \Sigma V := V \circ \iota.
    \]
    More generally, we write $\Sigma_a$ for the $a$-fold iterate of the shift functor.
\end{defn}

\begin{thm}\cite{N}\label{shiftThm}
    If $V$ is a finitely generated FI-module, then for all $a \gg 0$, the module $\Sigma_a V$ is free.
\end{thm}

Theorem \ref{shiftThm} leads naturally into a collection of important homological invariants of finitely generated FI-modules. These invariants have appeared in various forms across the literature \cite{CEF,CEFN,CMNR,LR,R}. Our exposition will mostly follow \cite{CMNR}.

\begin{defn}
    Let $V$ be a finitely generated FI-module. Then the \textbf{local degree} of $V$ is the largest integer $l(V)$ such that $\Sigma_{l(V)}V$ is \emph{not} free. In other words, $l(V)+1$ is the smallest integer for which the shift is free. If $V$ is itself a free-module then we set $l(V) = -1$. Similarly, the \textbf{stable degree} of $V$, denoted $s(V)$, is the generating degree of $\Sigma_{l(V)+1}V$.
\end{defn}

\begin{rmk}
    Because it will be useful to us later, we make the following observation about the stable degree. The free module $M(W)$ has the property that the dimension $\dim M(W)_n$ is in agreement with a polynomial in $n$. Indeed,
    \[
    \dim M(W)_n = \binom{n}{m}\dim W
    \]
    Importantly, the degree of this polynomial is $m$, which agrees exactly with the generating degree of $M(W)$. It follows that if you can bound the dimensions $\dim V_n$ of a finitely generated FI-module by a polynomial of some degree $d$, then $s(V) \leq d$. This perspective was the original perspective taken by \cite{CMNR}.
\end{rmk}

Putting it all together, our observations earlier about when free modules begin to display multiplicity stability, as well as Theorem \ref{shiftThm}, imply the following:

\begin{thm}
    If $V$ is a finitely generated FI-module, then the collection $\{V_n\}$ begins to display multiplicity stability no later than $n = l(V)+2s(V) + 1$
\end{thm}

As observed in \cite{CMNR}, the benefit of bounding our stable ranges with these particular homological invariants is that they behave rather nicely in spectral sequences. 

\begin{thm}\cite[Proposition 3.3]{CMNR}\label{technicalbound}
Let $f:M \rightarrow N$ be a map of finitely generated FI-modules. Then one has
\[
l(C) \leq \max\{2s(M)-2,l(M),l(N)\},
\]
where $C$ is either the kernel or cokernel of $f$.
\end{thm}

Particularly relevant for us will be the case wherein our spectral sequence degenerates immediately. That is to say, the case of a chain complex.

\begin{cor}
    Let $M^\bullet$ denote a homologically graded chain complex of FI-modules. Then one has,
    \[
    l(H_i(M^\bullet)) \leq \max\{2s(M^i)-2, 2s(M^{i+1})-2,l(M^{i+1}),l(M^i),l(M^{i-1})\}.
    \]
\end{cor}

\begin{proof}
    For this proof only, we make the following definitions
    \[
        f^i:M^i\rightarrow M^{i-1},\quad 
        g^i: \ker(f^i)\rightarrow M^i, \quad
        h^i: \im(f^{i+1}) \rightarrow \ker(f^i).
    \]

    According to Theorem \ref{technicalbound}, we have
    \begin{align*}
            l(H_i(M^\bullet)) = l(\coker(h^i)) &\leq \max\{2s(\im(f^{i+1}))-2,l(\im(f^{i+1})),l(\ker(f^i))\\
            &\leq \max\{2s(M^{i})-2,l(\im(f^{i+1})),l(\ker(f^i))\}.
    \end{align*}

    We also find, again applying Theorem \ref{technicalbound},

    \[
        l(\ker(f^{i})) \leq \max\{2s(M^i)-2,l(M^i),l(M^{i+1})\}
    \]

    As the above is true for all $i$, we may also conclude,
    \begin{align*}
        l(\im(f^{i+1})) = l(\coker(g^{i+1})) &\leq \max\{2s(\ker(f^{i+1})-2,l(\ker(f^{i+1})),l(M^{i+1})\}\\
        &\leq \max\{2s(M^{i+1})-2,l(\ker(f^{i+1})),l(M^{i+1})\} 
    \end{align*}
Our corollary now follows.

\end{proof}

To conclude this section, we take the time to record one final bound on when multiplicity stability begins. This bound, due to the originating work on FI-modules \cite{CEF} is not as general as the above homological work, as it presupposes two things. Firstly, it will only hold when one works over a field of characteristic 0. While it is the case that we are working in this context in this paper, it will unfortunately not be usable if one were to try to use our quantitative techniques to prove facts about the torsion groups appearing in the homologies of graph configuration spaces. Secondly, it is most useful in cases wherein you are computing stable ranges of homology modules coming from chain complexes entirely comprised of free modules. While this will be the case in what follows, it may not be the case in any future work that builds off what we do here.

\begin{thm} \cite[Proposition 3.3.3 and Lemma 6.3.2]{CEF} \label{relationBound}
Let
\[
\cdots \rightarrow M^2 \stackrel{f}\rightarrow M^1 \stackrel{g}\rightarrow M^0 \rightarrow \cdots
\]
be a chain complex of FI-modules over a field of characteristic 0, and assume that each of the modules $M^i$ is free. Then the homology group $\ker(g)/\im(f)$ achieves multiplicity stability no later than $n = \max\{s(M^2),s(M^1)\} + s(M^1)$. More generally, for any fixed partition $\lambda$ of some integer $m$, the multiplicity of $
\Sp(\lambda[n])$ appearing in $\ker(g)/\im(f)$ is constant for $n \geq \max\{s(M^2),s(M^1)\} + m$.
\end{thm}

\subsection{Configuration spaces of graphs}\label{section:confbackground}

\begin{defn}
Throughout this work, a \textbf{graph} will always refer to a finite CW-complex of dimension at most 1. Given a graph $G$, and an integer $n \geq 0$, the \textbf{$n$-pointed (ordered) configuration space} on $G$ is the topological space,
\[
\conf_n(G) = \{(x_1,\ldots,x_n) \in G^n \mid x_i \neq x_j\}.
\]
Relatedly, the \textbf{$n$-pointed unordered configuration space} is the quotient space
\[
\uconf_n(G) = \conf_n(G)/S_n,
\]
where the symmetric group acts properly discontinuously on $\conf_n(G)$ by permuting coordinates.
\end{defn}

Configuration spaces of higher dimensional manifolds have been a mainstay in algebraic topology since at least the work of Arnold \cite{Arn}. Configuration spaces of graphs are a much more modern concern, whose study has been in large part driven by applications to robotics \cite{abrams-ghrist,FarBook,Ghr}, as well as physics \cite{Mac,MS}. One feature that is common to much of the work on the subject in the last few years, is that it appears that explicit computations of the homology groups of these spaces is quite difficult \cite{ADCK,AK,BF,CL,Drum,Farl,FH}. To summarize some things that are currently known:

\begin{itemize}
    \item The first homology group $H_1(\uconf(G))$ is either torsion free, or otherwise can only contain 2-torsion. Whether or not this group has torsion is equivalent to whether or not $G$ is planar \cite{KP};
    \item The Betti numbers of $H_i(\uconf_n(G))$ and $H_i(\conf_n(G))$ have been computed for all $i$ and $n$, provided that $G$ is a tree \cite{CL,Farl};
    \item The Betti numbers of $H_i(\conf_2(G))$ have been computed for a number of graphs, including complete graphs and complete bipartite graphs, as well as all planar graphs \cite{BF,FH};
    \item The groups $H_2(\uconf_n(G))$ are computed in principal whenever $G$ is a planar graph \cite{AK};
    \item Betti numbers for a large collection of configuration spaces of small graphs with a small number of points have been accomplished through computer calculation in \cite{Drum}.
\end{itemize}

Notably, much of the work that has been completed in graph configuration spaces over the last few years has focused especially on the unordered case \cite{ADCK,AK,Farl,FS,R2,Swia}. In fact, when it comes to the ordered configuration spaces of graphs, very little is explicitly computed beyond the two particle case. This is due, in no small part, to the combinatorial explosion that occurs in these Betti numbers when $n$ gets large. 
For this reason, there has been a push in the literature over the last few years to understand the behavior of these spaces when the number of points is fixed, but the graph itself is allowed to vary in some nice family \cite{KR,Lut,MR,R3,RW}. This is where representation stability enters the picture.

\begin{defn}\label{def: FIgraph}
    By an \textbf{FI-graph}, we mean a functor $G_\bullet$ from FI to the category of graphs and graph homomorphisms. We say that an FI-graph is \textbf{finitely generated} if for all $n \gg 0$, and all $v \in V_{G_{n+1}}$, there exists some $v' \in V_{G_n}$ and an injection $f:[n] \hookrightarrow [n+1]$ such that $G_f(v') = v$.
\end{defn}

Most natural examples of FI-graphs that one might consider are finitely generated. For instance, complete graphs $K_n$ as well as the complete bipartite graphs $K_{n,n}$. Combinatorially, one reason such families are nice is that their many counting invariants behave quite regularly in $n$. For instance, we have the following.

\begin{thm}\cite[Theorem B]{RW}
    Let $G_\bullet$ be a finitely generated FI-graph, and let $H$ be any fixed graph. Then the number of copies of $H$ appearing as a subgraph within $G_n$ eventually agrees with a polynomial. In particular, both the number of vertices in $G_n$ as well as the number of edges in $G_n$ are in eventual agreement with polynomials.
\end{thm}

What is most relevant for us is the following theorem.

\begin{thm} \cite[Theorem G]{RW} \label{Gconffg}
    Let $G_\bullet$ denote a finitely generated FI-graph. Then for each $i,n \geq 0$ the FI-modules
    \[
    H_i(\conf_n(G_\bullet)) \text{ and } H_i(\uconf_n(G_\bullet)),
    \]
    are finitely generated.
\end{thm}

This theorem, along with the technical work of the previous section, is the spine of the present work. Notably, Theorem \ref{Gconffg} tells us absolutely nothing about the local or stable degrees of these modules. To apply our quantitative methods, we therefore need to get a better grasp of these quantities. This understanding will ultimately fall out of the topology of the situation, which we will see shortly.

Before getting to this topological work, we note that in one particular case the stable degree has already been computed quite precisely.

\begin{thm}\cite[Theorem C]{Lut}\label{specialCase}
    Let $(G,v)$ and $(H,w)$ be pairs of a graph with a choice of vertex. For each $k \geq 0$, we define $G_k^H$ to be the graph that is formed by wedging $k$ copies of $H$ to $G$ along the specified vertices. This turns the collection $G_\bullet^H$ into a finitely generated FI-graph, where the symmetric group acts by fixing $G$ and permuting the copies of $H$. Then for any $i,n \geq 0$ the FI-modules
    \[
    H_i(\conf_n(G^H_\bullet)) \text{ and } H_i(\uconf_n(G^H_\bullet)),
    \]
    each have stable degree $\leq n+3$.
\end{thm}

Importantly, one should note that in these cases the stable degree does not depend on $i$. This feature seems very much unique to this case, as we will see later.

If we assume that $G = v$ is a single vertex, and $H$ is a single edge, the resulting graphs $G^H_k$ are what are known as \textbf{star graphs}. Star graphs play a prominent role in the study of graph configuration spaces, as their presence within larger graphs has quite a bit of control on the behavior of the homology groups (see \cite{ADCK} or \cite{CL}, for instance).

Moving on from this case, it remains to develop a better understanding of the local and stable degrees in the cases of more general FI-graphs. To accomplish this, we have to lean on a number of discrete models of these spaces.

\begin{defn}
    Let $G$ be a graph, and $n \geq 0$. The topological space $G^n$ is naturally a CW complex with cellular structure induced from $G$. In particular, the cells all take the form $\sigma_1 \times \ldots \times \sigma_n$, where $\sigma_i$ is either a vertex or edge of $G$ for each $i$. With this in mind, we define the \textbf{$n$-pointed discretized configuration space of G}, $D_n(G)$, to be the subcomplex of $G^n$ formed by the cells of the form $\sigma_1 \times \ldots \times \sigma_n$, where for any $i \neq j$, the vertices contained in $\sigma_i$ and $\sigma_j$ do not overlap. We similarly define the \textbf{n-pointed unordered discretized configuration space of $G$}, $D^U_n(G)$, to be the quotient of $D_n$ by the action of $S_n$.
\end{defn}

While one would like for it to be the case for $\conf_n(G)$ and $D_n(G)$ to be homotopy equivalent to one another, this is not quite the case. For instance, $D_n(G)$ is the empty complex if the number of vertices of $G$ is much less than $n$. There is, however, a way to fix this problem at the cost of drastically increasing the number of cells.

\begin{thm}\cite{Abr}\label{AbramsModel}
    If $G$ is a graph such that,
    \begin{itemize}
        \item every path between distinct vertices in $G$ of degree not equal to 2 has at least $n-1$ edges and,
        \item every cycle within $G$ has at least $n+1$ edges,
    \end{itemize}
    then $D_n(G)$ is homotopy equivalent to $\conf_n(G)$. This equivalence is equivariant in the action of $S_n$, so that it descends to an equivalence of the unordered spaces as well.
\end{thm}

If $G$ is any graph, then we can always achieve the conditions of the above theorem by subdividing edges of $G$ with extra vertices. As subdivision is not seen by the configuration space, we can conclude that one may always work with discretized configuration spaces at the cost of adding more cells to $G^n$. Finally, observe that if $G$ is a simple graph and $n = 2$, then subdivision of this kind is never necessary.

The presence of the discretized configuration spaces allows us to immediately place bounds on our stable degrees.

\begin{prop}
    Let $G_\bullet$ denote a finitely generated FI-graph, and assume that $|V_{G_n}|$ is in eventual agreement with a polynomial of degree $d_V$, while $|E_{G_n}|$ is in eventual agreement with a polynomial of degree $d_E$. Then for any $k \geq 2$ and $i \leq k$ the stable degrees of the FI modules
    \[
    H_i(\conf_k(G_\bullet)) \text{ and } H_i(\uconf_k(G_\bullet)),
    \]
    is at most $i \cdot d_E + (k-i)(d_V + d_E)$. If $k = 2$, then this case be improved to $i \cdot d_E + (k-i)d_V$
\end{prop}

\begin{proof}
    The homology groups of $\uconf_k(G)$ are precisely the $S_k$-coinvariants of the homology groups of $\conf_k(G)$. Therefore, it suffices to prove the proposition for the ordered space $\conf_k(G)$. In this case we may use the fact that $\conf_k(G)$ is homotopy equivalent to $D_k(G)$, so long as $G$ is sufficiently subdivided, which is a cubical complex. The cubical $i$-chains of this complex grow no faster than polynomials of the necessary degree. If $k = 2$, then subdivision of the graph is unnecessary, improving the bounds in the stated way.
\end{proof}

\begin{rmk}\label{rmk: H0 is triv}
    The bound of the above proposition is not always sharp. For instance, because these spaces are connected so long as the graph has at least one vertex of degree at least 3, $H_0$ has stable degree 0. Moreover, the stable degrees given in Theorem \ref{specialCase} will generally be better in the case of $G_\bullet^H$.
\end{rmk}

In this work, we will be proving theorems - using the homological techniques of Section \ref{section:repstab}, the topological facts of the present Section \ref{section:confbackground}, and computer algebra - about the homology groups of $\conf_k(G)$ and $\uconf_k(G)$ for a number of families of graphs and $k \leq 3$. 

The cases of configuration spaces with only two points were treated in much detail by Barnett, Farber, and Hanbury in the works \cite{BF,FH}. In that work, the three authors provide formulas for the Betti numbers of ordered configuration spaces of a number of graphs, including the complete graph and the complete bipartite graphs.  We record one particularly important theorem from these works now.

\begin{thm}\cite[Propositions 21 and 22]{FH}\label{thm: H1 of conf2 of Knn}
Let $G_n$ denote either the complete graph on $n$ vertices, or the complete bipartite graph $K_{n,n}$. Then the natural inclusion 
\[
\conf_2(G_n) \hookrightarrow G_n \times G_n
\]
induces an isomorphism of $S_2 \times S_n$-representations,
\[
H_1(\conf_2(G_n)) \cong H_1(G_n \times G_n) \cong H_1(G_n) \oplus H_1(G_n),
\]
Provided that $n \geq 5$ if $G_n = K_n$, or $n \geq 3$ if $G_n = K_{n,n}$.
\end{thm}

Our Theorems \ref{thm: ordered complete} and \ref{thm: ordered bipartite} can be deduced from the above theorem and knowing the $S_n$-equivariant Euler characteristics, but we use computer algebra for a uniform treatment as the other families of graphs.

\section{The case of two particles}

In this section we provide explicit computations of the rational homology groups of ordered and unordered configuration spaces on two points for a number of FI-graphs. We drop the coefficient $\Q$ in the notations. The families we consider are the star or star tree graphs with $n$ leaves, the complete graphs on $n$ vertices, the complete bipartite graphs on $n+n$ vertices, the crown graphs on $n+n$ vertices, the complete tripartite graphs on $n+n+1$ vertices, and the Kneser graphs $K(n,2)$. For each family, the analysis will follow the same steps:
\begin{enumerate}
    \item computing the stability bounds,
    \item deriving theoretically the $S_n$-representations afforded by the chain modules of the discretized configuration spaces in both ordered and unordered cases,
    \item computing using software the $S_n$-equivariant degree-2 homology groups in both ordered and unordered cases, and
    \item analyzing generators of selected subrepresentations in the stable second homology in the ordered configuration space and giving them combinatorial interpretations.
\end{enumerate}

\subsection{The star graph}
\paragraph{\textit{Computing the bounds}}
For any finite set $A$, let $X_A$ be the star graph. That is to say, the tree with a single internal vertex $v_0$ connected to leaves labeled in $A$. Then $\{X_A\}$ is a family of $S_A$-graphs where $S_A$ permutes the leaves. In particular, $X_n = G^H_n$, where $G$ is a single vertex and $H$ is an edge.

In this case, the discretized configuration space $D_2(X_A)$ is 1-dimensional. The 0-cells are in bijection with pairs of the form $i \times j$, with $i \neq j \in A \cup \{0\}$ where the symmetric group acts in the usual way on elements in $A$ and acts trivially on $0$. In particular, if we write $C_0$ for the module of 0-cells, there is an isomorphism
\[
M(2) \oplus M(1) \oplus M(1) \cong C_0.
\]
This correspondence sends an injection $f:[2] \rightarrow A$, representing a basis vector of $M(2)_A$, to the 0-cell $f(1) \times f(2)$. It will send an injection $g:[1] \rightarrow A$, representing a basis vector of $M(1)_A$ to either the zero cell $g(1) \times 0$ or $0 \times g(1)$, depending on which copy of $M(1)$ it originated from. In particular, $s(C_0)=  2$

Moving on the the 1-cells $C_1$, we claim that this module is also free. To see this we use Theorem \ref{sharpisfree}. Let $f:[n] \sqcup \{\infty\} \rightarrow [m] \sqcup \{\infty\}$ be a morphism in FI$\sharp$ with true domain $A \subseteq [n]$, and let $c = \sigma_1 \times \sigma_2$ be a 1-cell. Looking at all of the non-zero numbers that appear in the cells $\sigma_i$, if they are all contained in the true domain of $f$, then $f$ will act on this element in the same way that the map induced by
\[
f|_A: X_A \rightarrow X_m
\]
acts on the corresponding 1-chains. If there are non-zero numbers in $c$ outside of $A$, then $f$ will act as the zero map. As with the 0-chains, we can say $s(C_1) = 2$

\begin{rmk}
    We will see in future sections that actions by FI$\sharp$ arise in similar fashion for a lot of families of FI-graphs. The key point here is that given a subset $A \subseteq [n]$, the vertices of $X_n$ contained in $A$ form an induced subgraph that is isomorphic to $X_{A}$. Similar behaviors will be seen in all of the FI-graph families we consider.
    
    We therefore take the time here to point out that while the chain groups often times end up being free modules, it is almost never the case that this FI$\sharp$ structure commutes with the differential of the chain complex. Indeed, this is because a chain can contain an edge whose one end point is in the true domain of the FI$\sharp$ morphism, while the other is not. The definition of our action will send this chain to 0, whereas its image under the differential will not be sent to zero under this action. In other words, while the chains of our complex are free, we do not in general expect the homology groups to be free.
\end{rmk}

\begin{thm} \label{starBound}
Let $X_n$ denote the star graph on $n$-leaves. Then the FI-modules $H_1(\conf_2(X_n))$ and $H_1(\uconf_2(X_n))$ are multiplicity stable no later than $n = 4$.
\end{thm}

\paragraph{\textit{Computing the stable homology}}

\begin{thm}\label{thm: star}
For star graphs $X_n$, we have the following isomorphisms of $S_n$-representations:
\begin{enumerate}
    \item $H_1(\conf_2(X_n)) \cong \mathrm{Sp}_n(1,1) \oplus \mathrm{Sp}_n(2)$ for $n\geq 4$.
    \item $H_1(\uconf_2(X_n)) \cong \mathrm{Sp}_n(1,1)$ for $n\geq 4$.
\end{enumerate}
\end{thm}

\begin{proof}
By Theorem \ref{starBound}, it suffices to compute $H_1(\conf_2(X_n))$ and $H_1(\uconf_2(X_n))$ for $n=4$. We first compute the $S_4$-representations afforded by the cellular chain modules $C_0$ and $C_1$ of $D_2(X_4)$. They are
\begin{align*}
    C_0 &\cong (2,1,1) \oplus (2,2) \oplus 4(3,1) \oplus 3(4) \\
    C_1 &\cong 2(2,1,1) \oplus 2(2,2) \oplus 4(3,1) \oplus 2(4).
\end{align*}
By Remark \ref{rmk: H0 is triv}, the zeroth homology $H_0(\conf_2(X_n)) \cong \Q$ is the trivial representation. The Euler characteristic $\chi(\conf_2(X_n))$, as a virtual $S_4$-representation over $\Q$, satisfies \[\chi = \sum_{i\geq0}(-1)^i C_i \cong \sum_{i\geq0}(-1)^i H_i.\] By Euler characteristic, we obtain $H_1(\conf_2(X_4)) \cong (2,1,1) \oplus (2,2)$. 

The unordered case is analogous. The chain modules of $D^U_2(X_4)$ are 
\begin{align*}
    C_0 &\cong (2,2) \oplus 2(3,1) \oplus 2(4) \\
    C_1 &\cong (2,1,1) \oplus (2,2) \oplus 2(3,1) \oplus (4).
\end{align*}
\end{proof}

Note that it is possible to arrive at Theorem \ref{thm: star} without knowledge of multiplicity stability; however, knowing the stable range reduces the proof to computing an example.

\subsection{The complete graph}

\paragraph{\textit{Computing the bounds}}

Write $V_{K_n}=[n]$ and ${ij}$ for the edge between vertices $i$ and $j$. We assume edges are oriented from $i$ to $j$ for $i<j$. Just as with the star graphs, we consider the chains of the discretized configuration space $D_2(K_n)$. Then $C_2$ is generated by cells of the form ${ij}\times{kl}$ for distinct $i,j,k,l\in[n]$, subject to the relations ${ij}\times{kl} = -{ji}\times{kl} = -{ij}\times{lk}$. The chain module $C_1$ is generated by cells of the form ${ij}\times k$ and $k \times {ij}$ for distinct $i,j,k$, subject to the relation ${ij}\times k = -{ji}\times k$ and $k\times {ij}= -k\times {ji}$. The chain module $C_0$ is generated by cells of the form $i\times j$ for distinct $i,j$. Similar to the case of star graphs, these chain modules are always free. To see this, we may impose an FI$\sharp$-structure on these groups by having morphisms act in the natural way on any chain that only contains numbers from their true domain, and as 0 otherwise. Using Theorem \ref{relationBound}, along with the fact that $s(C_2) = 4, s(C_1) = 3, s(C_0) = 2$, we obtain the following.

\begin{thm} \label{completeBound_2pts}
Let $K_n$ denote the complete graph on $n$ vertices. Then the FI-modules $H_1(\conf_2(K_n);\Q)$ and $H_2(\conf_2(K_n))$ are multiplicity stable no later than $n = 7$ and $n = 8$, respectively. The same holds for $\uconf_2(K_n)$.
\end{thm}

We now give an example that demonstrates the steps of the analysis before moving on to discuss the general case.

\begin{example}[The complete graph $K_4$]
    First, we compute the cellular chain modules of $D_2(K_4)$ as representations:
    \begin{align*}
        C_2 &= (1, 1, 1, 1) \oplus (2, 1, 1) \oplus (2, 2)\\
        C_1 &= 2(1, 1, 1, 1) \oplus 4(2, 1, 1) \oplus 2(2, 2) \oplus 2(3, 1)\\
        C_0 &= (2, 1, 1) \oplus (2, 2) \oplus 2(3, 1) \oplus (4)
    \end{align*}
    The boundary of a 2-cell ${ij}\times {kl}$ is given as below
    \begin{align*}
        \partial({ij}\times {kl}) &= \partial({ij})\times {kl}-{ij}\times \partial({kl})\\
        &= i\times {kl}-j\times {kl}-{ij}\times k+{ij}\times l.
    \end{align*}
    Since $\dim C_2 = 6$ and $\dim C_1 = 24$, we can represent the differential $d_1: C_2 \to C_1$ as a $24 \times 6$ matrix. Using software, we determine that $d_1$ has full rank. Therefore $H_2 = \ker d_1 \cong 0$. Again, $H_0 \cong \Q$ is the trivial representation. 
    We deduce $H_1$ through the Euler characteristic:
    \begin{align*}
        H_0-H_1+H_2 &= C_0-C_1+C_2\\
        H_1 &= (1,1,1,1) \oplus 2(2,1,1)
    \end{align*}
\end{example}

\paragraph{\textit{$S_n$-equivariant chain modules}} We now give a general description for the chain modules of $D_2(K_n)$ as $S_n$-representations.

\begin{lem}
For $n\geq 4$, the $S_n$-equivariant chain modules of $D_2(K_n)$ are
\begin{align*}
C_2(D_2(K_n)) &\cong M(\ind_{S_2\times S_2}^{S_4} (1,1) \otimes (1,1)) \\
C_1(D_2(K_n)) &\cong 2M(\ind_{S_2}^{S_3} (1,1)) \\
C_0(D_2(K_n)) &\cong M(2) 
\end{align*}
\end{lem}

\begin{proof}
Recall that $C_2$ is generated by cells of the form ${ij}\times{kl}$ for distinct $i,j,k,l\in[n]$, subject to the relations ${ij}\times{kl} = -{ji}\times{kl} = -{ij}\times{lk}$. The submodule generated by cells for $n=4$ is \[\ind_{S_2\times S_2}^{S_4}(1,1)\otimes (1,1) = (1,1,1,1)\oplus (2,1,1)\oplus (2,2)\] as an $S_4$-representation. When $n>4$, the action from $S_{\{5,\dots,n\}}$ on this submodule is trivial, so $C_2$ has description \[\ind_{S_4\times S_{n-4}}^{S_n}(\ind_{S_2\times S_2}^{S_4}(1,1)\otimes (1,1))\otimes \Q,\] which, in our notation, is exactly $M(\ind_{S_2\times S_2}^{S_4}(1,1)\otimes (1,1))$. Similarly, we deduce the $S_n$-representation structures of $C_1$ and $C_0$ as induced representations.
\end{proof}

\begin{lem}
The representations afforded by chain modules of $D^U_2(K_n)$ are
\begin{align*}
    C_2(D^U_2(K_n)) &\cong M((2,1,1)) \\
    C_1(D^U_2(K_n)) &\cong M(\ind_{S_2}^{S_3}(1,1))\\
    C_0(D^U_2(K_n)) &\cong M((2)).
\end{align*}   
\end{lem}

\begin{proof}
Using the same notation as in the ordered case, the module $C_2$ is generated by cells of the form ${ij}\times{kl}$ for distinct $i,j,k,l\in [n]$, subject to the relations ${ij}\times {kl} = -{ji}\times {kl}=-{kl}\times {ij}$. When $n=4$, we have $C_2 \cong (2,1,1)$. The module $C_1$ is generated by cells of the forms ${ij}\times k$ and $k\times {ij}$ for distinct $i,j,k\in[n]$, subject to the relations ${ij}\times k = -{ji}\times k$ and ${ij}\times k = k\times {ij}$. The chain module $C_0$ is generated by cells $i \times j$ for distinct $i,j\in [n]$, with relation $i \times j=j\times i$.
\end{proof}

\paragraph{\textit{Stable homology groups}}
\begin{thm}\label{thm: ordered complete}
    We have isomorphisms of $S_n$-representations,
    \begin{align*}
        &H_2(\conf_2(K_n);\Q)\cong\mathrm{Sp}_n(1,1,1,1) \oplus \mathrm{Sp}_n(2,1,1) \oplus \mathrm{Sp}_n(2,2),\quad n\geq 6\\
        &H_1(\conf_2(K_n);\Q)\cong2\mathrm{Sp}_n(1,1),\quad n\geq 5.
    \end{align*}
\end{thm}

\begin{proof}
Combine Theorem \ref{completeBound_2pts} and Table \ref{table: ordered 2 pts Kn}.
\end{proof}

\begin{remark}
    This theorem indicates that the true stable range of the second homology group is actually $n \geq 6$. We will see in later sections that the bounds given by Theorem \ref{relationBound} will be tight in the cases of the complete bipartite graph and the crown graph.
\end{remark}


\begin{table}[th]
\caption{$S_n$-representations afforded by the rational homology of the ordered configuration space of two particles on the complete graph $K_n$ for $4\leq n \leq 8$.}
\label{table: ordered 2 pts Kn}
\begin{tabular}{@{}cclc@{}} \toprule
$n$ & $i$ & $H_i(\conf_2(K_n))$ & Time (s) \\
\midrule
\multirow[t]{2}{*}{4} & 2 & 0 & $<1$ \\[2pt]
& 1 & $(1,1,1,1) + 2(2,1,1)$ & \\ \midrule
\multirow[t]{2}{*}{5} & 2   & $(1,1,1,1,1)$ & $<1$ \\[2pt]
& 1   & $2(3,1,1)$ & \\\midrule
\multirow[t]{2}{*}{6} & 2 & $(2, 1, 1, 1, 1) + (2, 2, 1, 1) + (2, 2, 2)$ & $<1$ \\[2pt]
& 1 & $2(4,1,1)$ & \\\midrule
\multirow[t]{2}{*}{7} & 2 & $(3, 1, 1, 1, 1) + (3, 2, 1, 1) + (3, 2, 2)$ & 3 \\[2pt]
& 1 & $2(5,1,1)$ & \\\midrule
\multirow[t]{2}{*}{8} & 2 & $(4, 1, 1, 1, 1) + (4, 2, 1, 1) + (4, 2, 2)$ & 72 \\[2pt]
& 1 & $2(6,1,1)$ & \\
\bottomrule
\end{tabular}

\end{table}


\begin{table}[th]
\caption{$S_n$-representations afforded by the rational homology of the unordered configuration space of two particles on the complete graph $K_n$ for $4\leq n \leq 8$.}
\label{table: unordered 2 pts on Kn}
\begin{tabular}{@{}cclc@{}} \toprule
$n$ & $i$ & $H_i(\uconf_2(K_n))$ & Time (s) \\
\midrule
\multirow[t]{2}{*}{4} & 2 & 0 & $<1$ \\[2pt]
& 1 & $(1,1,1,1) + (2,1,1)$ & \\ \midrule
\multirow[t]{2}{*}{5} & 2   & 0 & $<1$ \\[2pt]
& 1   & $(3,1,1)$ & \\\midrule
\multirow[t]{2}{*}{6} & 2 & $(2, 2, 1, 1)$ & $<1$ \\[2pt]
& 1 & $(4,1,1)$ & \\\midrule
\multirow[t]{2}{*}{7} & 2 & $(3, 2, 1, 1)$ & 3 \\[2pt]
& 1 & $(5,1,1)$ & \\\midrule
\multirow[t]{2}{*}{8} & 2 & $(4, 2, 1, 1)$ & 18 \\[2pt]
& 1 & $(6,1,1)$ & \\
\bottomrule
\end{tabular}

\end{table}

\begin{thm}
    We have isomorphisms of $S_n$-representations 
    \begin{align*}
        &H_2(\uconf_2(K_n))\cong \mathrm{Sp}_n(2,1,1),\quad n\geq 6\\
        &H_1(\uconf_2(K_n))\cong \mathrm{Sp}_n(1,1),\qquad n\geq 5.
    \end{align*}
\end{thm}

\paragraph{\textit{Combinatorial interpretations}}
In this step of the analysis, we give combinatorial interpretations of subrepresentations of $H_2(\conf_2(K_n))$. In the study of graph configuration spaces, many results rely on understanding generators of $H_i$, see, for example, the discuss in the introduction of \cite{ADCK2}. By scrutinizing generators of subrepresentations, we hope to lay the foundation for gaining similar understandings when $S_n$-equivariant structures are present.

An important type of generators for $H_i$ is the \textit{product class}. We recall its definition.

\begin{defn}\cite[Definition 1.2]{CL}
    A homology class $\sigma\in H_q(\conf_n(G))$ is called the product of classes $\sigma_1\in H_{q_1}(\conf_{n_1}(G_1))$ and $\sigma_2\in H_{q_2}(\conf_{n_2}(G_2))$ for $q_1+q_2=q$ and $n_1+n_2=n$ if it is the image of $\sigma_1\otimes \sigma_2$ under the map \[H_q(\conf_n(G_1\sqcup G_2))\to H_q(\conf_n(G))\] induced by an embedding $G_1\sqcup G_2\hookrightarrow G$.
\end{defn}

In particular, if $G_1$ and $G_2$ are homeomorphic to the circle $S^1$ and $q_1=q_2=1$, then we call their product class a \textit{toric class}. Concretely, we can represent a toric class in $H_2(D_2(G))$ as follows. If $G_1$ is a cycle $v_1-v_2-\cdots-v_k-v_1$ and $G_2$ is a cycle $w_1-w_2-\cdots-w_\ell-w_1$ that is disjoint from $G_1$, then the element
\[\sum_{i=1}^k \sum_{j=1}^\ell \overline{v_iv_{i+1}}\times\overline{w_jw_{j+1}}\] is the corresponding toric class in $H_2(D_2(G))$, where $\overline{v_iv_{i+1}}$ is the edge $v_i\to v_{i+1}$ and $v_{i+1}=v_1$ if $i=k$, and similarly for the $w$'s.

The following proposition gives a combinatorial interpretation of the irreducible components of $H_2(\conf_2(K_n))$. As a consequence, it explains the stabilization of $\mathrm{Sp}_n(1,1,1,1)$ at $n=5$ and that of $\mathrm{Sp}_n(2,1,1)$ and $\mathrm{Sp}_n(2,2)$ only at $n=6$. In this proposition, we write $C_3$ for the cycle on three vertices.

\begin{prop}
    In $H_2(\conf_2(K_n))$, the submodule $\mathrm{Sp}_n(1,1,1,1)$ is induced by $K_5\hookrightarrow K_n$, the submodule $\mathrm{Sp}_n(2,1,1)$ is generated by toric classes induced by $C_3 \sqcup C_3\hookrightarrow K_6$, and the submodule $\mathrm{Sp}_n(2,2)$ is induced by $C_3\sqcup C_3 \sqcup K_{3,3}\hookrightarrow K_6$.
\end{prop}

\begin{proof}
The proof is by explicit computations. Since $H_2(\conf_2(K_n))$ stabilizes at $n=6$, it suffices to understand the case $K_6$. First, we generate a basis for $H_2$ using 2-chains. Then we project $H_2$ to the irreducible submodules following, for example, \cite[Section 2.7]{Serre1977}. This projection is canonical when an irreducible subrepresentation has multiplicity 1. We denote by $V_\lambda$ the subspace of $H_2(\conf_2(K_n))$ affording the irreducible representation indexed by the partition $\lambda$.

The subspace $V_{(2,1,1,1,1)}$ is generated by the following element, which is defined on the induced subgraph on vertices $\{1,2,3,4,6\}$, and its $S_6$-images.{\footnotesize \begin{align*} 
&{12} \times {34}
- {12} \times {36}
+{12} \times {46}
- {13} \times {24}
+{13} \times {26}
- {13} \times {46}
+{14} \times {23}
- {14} \times {26}
+{14} \times {36}
- {16} \times {23}\\
+&{16} \times {24}
- {16} \times {34}
+{23} \times {14}
- {23} \times {16}
+{23} \times {46}
- {24} \times {13}
+{24} \times {16}
- {24} \times {36}
+{26} \times {13}
- {26} \times {14}\\
+&{26} \times {34}
+{34} \times {12}
- {34} \times {16}
+{34} \times {26}
- {36} \times {12}
+{36} \times {14}
- {36} \times {24}
+{46} \times {12}
- {46} \times {13}
+{46} \times {23}.
\end{align*}}The component $V_{(2,2,1,1)}$ is generated by the toric class $(125)\sqcup (346)$ and its $S_6$-images.

The component $V_{(2,2,2)}$ is generated by the following element and its $S_6$-images.
{\footnotesize \begin{align*}
2 (&{12} \times {45} 
-{12} \times {46} 
+{12} \times {56} 
-{13} \times {45} 
+{13} \times {46} 
-{13} \times {56} 
+{23} \times {45} 
-{23} \times {46} 
+{23} \times {56} \\
+&{45} \times {12}  
-{45} \times {13} 
+{45} \times {23} 
-{46} \times {12} 
+{46} \times {13} 
-{46} \times {23} 
+{56} \times {12} 
-{56} \times {13} 
+{56} \times {23}) \\
&\\
+&{14} \times {25} 
-{14} \times {26} 
-{14} \times {35} 
+{14} \times {36} 
-{15} \times {24} 
+{15} \times {26} 
+{15} \times {34} 
-{15} \times {36} 
+{16} \times {24}\\
-&{16} \times {25} 
-{16} \times {34} 
+{16} \times {35} 
-{24} \times {15} 
+{24} \times {16} 
+{24} \times {35} 
-{24} \times {36} 
+{25} \times {14} 
-{25} \times {16}\\
-&{25} \times {34} 
+{25} \times {36} 
-{26} \times {14} 
+{26} \times {15} 
+{26} \times {34} 
-{26} \times {35} 
+{34} \times {15} 
-{34} \times {16} 
-{34} \times {25}\\
+&{34} \times {26} 
-{35} \times {14} 
+{35} \times {16} 
+{35} \times {24} 
-{35} \times {26} 
+{36} \times {14} 
-{36} \times {15} 
-{36} \times {24} 
+{36} \times {25}.
\end{align*}}This class is the sum of two other classes. The first one is a toric class from $V_{(2,2,1,1)}$ defined on $(123)\sqcup (456)$ (dashed edges in Figure \ref{fig: K6}); and the second class is defined on the complement of these cycles, which is a complete bipartite graph $K_{3,3}$ (solid edges in Figure \ref{fig: K6}). It is worth noting that the class induced by $K_{3,3}\hookrightarrow K_6$ by itself is not in any of the three irreducible components.
\begin{figure}
\begin{tikzpicture}
\newdimen\r
\r=1cm
\coordinate (1) at (0:\r);
\coordinate (2) at (60:\r);
\coordinate (3) at (120:\r);
\coordinate (4) at (180:\r);
\coordinate (5) at (240:\r);
\coordinate (6) at (300:\r);

\draw (1) node[right] {$1$};
\draw (2) node[right] {$2$};
\draw (3) node[left] {$3$};
\draw (4) node[left] {$4$};
\draw (5) node[left] {$5$};
\draw (6) node[right] {$6$};

\draw (1)--(4);
\draw (1)--(5);
\draw (1)--(6);
\draw (2)--(4);
\draw (2)--(5);
\draw (2)--(6);
\draw (3)--(4);
\draw (3)--(5);
\draw (3)--(6);
\draw[dashed] (1)--(2)--(3)--cycle;
\draw[dashed] (4)--(5)--(6)--cycle;
\end{tikzpicture}
\caption{The edges of $K_6$ decomposes into two 3-cycles and $K_{3,3}$.}
\label{fig: K6}
\end{figure}
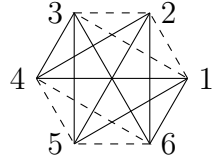
\end{proof}

\subsection{Complete bipartite graphs}

\paragraph{\textit{Computing the bounds}}
Let $K_{n,n}$ be the complete bipartite graph with $2n$ vertices. Label its vertices by $\{1,\dots,n\} \cup\{\bar 1,\dots,\bar n\}$. The edge set consists of $\{i,\bar j\}$ for every pair $i,j\in [n]$; we assume the edges are oriented $i\to \bar j$. The symmetric group $S_n$ acts on $K_{n,n}$ by permuting the vertices $\{1,\dots,n\}$ and $\{\bar1,\dots,\bar n\}$ simultaneously. We begin by computing bounds on when our stable behavior will kick in. Just as with the prior cases, we consider the chains of the discretized configuration space $D_2(K_{n,n})$. As before, the chain modules $C_i$ are always free. To see this, we may impose an FI$\sharp$-structure on these groups by having morphisms act in the natural way on any chain that only contains numbers from their true domain, and as 0 otherwise. Using Theorem \ref{relationBound}, along with the fact that $s(C_2) = 4, s(C_1) = 3, s(C_0) = 2$, we obtain the following.

\begin{thm} \label{completebipartiteBound}
Let $K_{n,n}$ denote the complete bipartite graph on $2n$ vertices. Then the FI-modules $H_1(\conf_2(K_{n,n}))$ and $H_2(\conf_2(K_{n,n}))$ are multiplicity stable no later than $n = 7$ and $n = 8$, respectively. The same holds for the unordered configuration spaces.
\end{thm}

\paragraph{\textit{$S_n$-equivariant chain modules}}

\begin{lem}
The representations afforded by chain modules of $D_2(K_{n,n})$ are
\begin{align*}
C_2(D_2(K_{n,n})) &\cong 2M(2) \oplus 4M(3) \oplus M(4),\quad n\geq 4\\
C_1(D_2(K_{n,n})) &\cong 8M(2) \oplus 4M(3),\quad n\geq 3\\
C_0(D_2(K_{n,n})) &\cong 2M(1) \oplus 4M(2),\quad n\geq 2 
\end{align*}
\end{lem}

\begin{proof}
Recall that the edges of $K_{n,n}$ are ${i\bar j}$ and are oriented as $i\to \bar j$. The chain module $C_2$ is generated by ordered pairs of edges, which partition into seven $S_n$-orbits, or combinatorial types; they are
\[
{i\bar i}\times {j\bar j}, \,{i\bar j}\times {j\bar i},\,
{i\bar j}\times {k\bar k}, \,{k\bar k}\times {i\bar j}, \,{i\bar j}\times {j\bar k}, \,{j\bar k}\times {i\bar j}, \,
{i\bar j}\times {k\bar l}.
\] Since the $S_n$-action on $K_{n,n}$ never reverses an edge, these generators satisfy no relations. Each type contributes a summand of $M(a)$ where $a$ is the number of letters appearing.

The chain module $C_1$ is generated by ordered pairs of an edge and a vertex. Assuming the edge is ordered first, the combinatorial types are
\[
{i\bar i} \times j,\,{i\bar i} \times \bar j,\,{i\bar j} \times \bar i,\,{i\bar j} \times j,\,{i\bar j} \times k,\,{i\bar j} \times \bar k.
\] Each type contributes a summand of $M(a)$ where $a$ is the number of letters appearing. The situation is analogous when the edge is ordered second.

Lastly, the chain module $C_0$ is generated by ordered pairs of vertices. Their combinatorial types are
\[
i\times \bar i,\,\bar i\times i,\,i\times j,\,i\times\bar j,\,\bar i\times j,\,\bar i\times \bar j.
\] Each type contributes a summand of $M(a)$ where $a$ is the number of letters appearing.
\end{proof}

\begin{lem}
The representations afforded by chain modules of $D^U_2(K_{n,n})$ are
\begin{align*}
C_2(D^U_2(K_{n,n})) &\cong 2M((1,1))\oplus 2M(3)\oplus M(2(2,1,1)\oplus 2(3,1)),\quad n\geq 4\\
C_1(D^U_2(K_{n,n})) &\cong 4M(2)\oplus 2M(3),\quad n\geq 3\\
C_0(D^U_2(K_{n,n})) &\cong M(1)\oplus 2M((2)) \oplus M(2),\quad n\geq 2
\end{align*}
\end{lem}
\begin{proof}
The combinatorial types in $C_2$ are
\[
{i\bar i}\times{j\bar j},\,{i\bar j}\times{j\bar i},\,{i\bar j}\times{k\bar k},\,{i\bar j}\times{j\bar k},\,{i\bar j}\times{k\bar l}.
\] The generators satisfy the relation $e_1\times e_2=-e_2\times e_1$ where $e_i$ is an edge. Respectively, they contribute the following summands
\[
M((1,1)),\,M((1,1)),\,M(3),\,M(3),\,M(2(2,1,1)\oplus 2(3,1)).
\] The combinatorial types in $C_1$ are \[
{i\bar i}\times j,\,{i\bar i}\times \bar j,\,{i\bar j}\times j,\,{i\bar j}\times \bar i,\,{i\bar j}\times k,\,{i\bar j}\times \bar k.
\] The generators satisfy the relation $e\times v=-v\times e$ where $e$ is an edge and $v$ is a vertex. The first four contribute $M(2)$ and the last two $M(3)$. 

The combinatorial types in $C_0$ are \[
i\times \bar i,\,i\times  j,\,i\times \bar j,\,\bar i\times \bar j.
\] They satisfy the relation $v_1\times v_2=v_2\times v_1$ where $v_i$ is a vertex. The first contributes $M(1)$, the second and fourth $M((2))$, the third $M(2)$.
\end{proof}

\paragraph{\textit{Stable homology}}


\begin{thm}\label{thm: ordered bipartite}
There are isomorphisms of $S_n$-representations
\begin{align*}
H_2(\conf_2(K_{n,n})) \cong &\mathrm{Sp}_n(1,1,1,1)\oplus 3\mathrm{Sp}_n(2,1,1) \oplus 2\mathrm{Sp}_n(2,2) \oplus 3\mathrm{Sp}_n(3,1) \oplus \mathrm{Sp}_n(4) \\
&\qquad \oplus 4\mathrm{Sp}_n(1,1,1) \oplus 8\mathrm{Sp}_n(2,1) \oplus 4\mathrm{Sp}_n(3) \oplus 6\mathrm{Sp}_n(1,1) \oplus 6\mathrm{Sp}_n(2) \\
&\qquad \oplus 4\mathrm{Sp}_n(1) \oplus 2\mathrm{Sp}_n(\emptyset), \quad n\geq 8\\
H_1(\conf_2(K_{n,n}))\cong & 2\mathrm{Sp}_n(1,1)\oplus 2\mathrm{Sp}_n(2)\oplus 2\mathrm{Sp}_n(1)\oplus 2\mathrm{Sp}_n(\emptyset),\quad n\geq 4
\end{align*}
\end{thm}

\begin{proof}
In Table \ref{table: H2Knn as SnxS2-reps}, the isotypical component of $H_2(\conf_2(K_{n,n}))$ indexed by $\lambda$ is $\lambda \otimes (1,1) \oplus \lambda \otimes (2)$. We use the Euler characteristic and that $H_0=\Q$ to deduce $H_1$.
\end{proof}

\begin{table}[th]
\caption{$S_n\times S_2$-representations afforded by the rational homology of the ordered configuration space of two particles on the complete bipartite graph $K_{n,n}$ for $3\leq n\leq 8$.}
\label{table: H2Knn as SnxS2-reps}
\begin{tabular}{@{}cp{0.80\textwidth}c@{}} \toprule
$n$ & $H_2(\conf_2(K_{n,n});\Q)$ & Time (s) \\
\midrule
3  & $(3)\otimes (1,1)$ & $<1$ \\\midrule
4  & $(1,1,1,1)\otimes (1,1) + (2,1,1)\otimes (1,1) + 2(2,2)\otimes (1,1) + (3,1)\otimes (1,1) + 2(4)\otimes (1,1) + 2(2,1,1)\otimes (2) + 2(3,1)\otimes (2)$ & $2$ \\\midrule
5  & $(1,1,1,1,1)\otimes (1,1) + 2(2,1,1,1)\otimes (1,1) + 3(2,2,1)\otimes (1,1) + 2(3,1,1)\otimes (1,1) + 3(3,2) \otimes (1,1) + 2(4,1)\otimes (1,1) + 2(5) \otimes (1,1)  + 2(2,1,1,1)\otimes (2) + 2(2,2,1)\otimes (2) + 4(3,1,1)\otimes (2) + 2(3,2) \otimes (2) + 2(4,1)\otimes (2)$ & 18 \\\midrule
6  & $(2,1,1,1,1)\otimes (1,1) + (2,2,1,1)\otimes (1,1) + 2(2,2,2)\otimes (1,1) + 2(3,1,1,1)\otimes (1,1) + 4(3,2,1)\otimes (1,1) + (3,3)\otimes (1,1) + 2(4,1,1)\otimes (1,1) + 4(4,2)\otimes (1,1) + 2(5,1)\otimes (1,1) + 2(6)\otimes (1,1) + 
        2(2,2,1,1)\otimes (2) + 2(3,1,1,1)\otimes (2) + 4(3,2,1)\otimes (2) + 2(3,3)\otimes (2) + 4(4,1,1)\otimes (2) + 2(4,2)\otimes (2) + 2(5,1)\otimes (2)$ & 3083  \\\midrule
7  & $(3,1,1,1,1)\otimes (1,1) + (3,2,1,1)\otimes (1,1) + 2(3,2,2)\otimes (1,1) +(3,3,1)\otimes (1,1)+ 2(4,1,1,1)\otimes (1,1) + 4(4,2,1)\otimes (1,1) + 2(4,3)\otimes (1,1) + 2(5,1,1)\otimes (1,1) + 4(5,2)\otimes (1,1) + 2(6,1)\otimes (1,1) + 2(7)\otimes (1,1) + 
2(3,2,1,1)\otimes (2) + 2(3,3,1)\otimes (2) + 2(4,1,1,1)\otimes (2) + 4(4,2,1)\otimes (2) + 2(4,3)\otimes (2) + 4(5,1,1)\otimes (2) + 2(5,2)\otimes (2) + 2(6,1)\otimes (2)$ & 20304  \\\midrule
8 & $(4, 1, 1, 1, 1) \otimes (1, 1) +
(4, 2, 1, 1) \otimes (1, 1) +
2(4, 2, 2) \otimes (1, 1) +
(4, 3, 1) \otimes (1, 1) +
(4, 4) \otimes (1, 1) +
2(5, 1, 1, 1) \otimes (1, 1) +
4(5, 2, 1) \otimes (1, 1) +
2(5, 3) \otimes (1, 1) +
2(6, 1, 1) \otimes (1, 1) +
4(6, 2) \otimes (1, 1) +
2(7, 1) \otimes (1, 1) +
2(8) \otimes (1, 1) +
2(4, 2, 1, 1) \otimes (2) +
2(4, 3, 1) \otimes (2) +
2(5, 1, 1, 1) \otimes (2) +
4(5, 2, 1) \otimes (2) +
2(5, 3) \otimes (2) +
4(6, 1, 1) \otimes (2) +
2(6, 2) \otimes (2) +
2(7, 1) \otimes (2)$ & 738573 \\
\bottomrule
\end{tabular}

\end{table}


\begin{thm}
We have the following isomorphisms of $S_n$-representations, for $n\geq 4$,
\begin{align*}
H_2(\uconf_2(K_{n,n})) \cong & 2\mathrm{Sp}_n(2,1,1) \oplus 2\mathrm{Sp}_n(3,1) \oplus 2\mathrm{Sp}_n(1,1,1) \oplus 4\mathrm{Sp}_n(2,1)\oplus 2\mathrm{Sp}_n(3) \\
& \qquad \oplus 4\mathrm{Sp}_n(1,1) \oplus 2\mathrm{Sp}_n(2)\oplus 2\mathrm{Sp}_n(1), \quad n\geq 7\\
H_1(\uconf_2(K_{n,n}))\cong & \mathrm{Sp}_n(1,1)\oplus \mathrm{Sp}_n(2)\oplus \mathrm{Sp}_n(1)\oplus \mathrm{Sp}_n(\emptyset),\quad n\geq 4
\end{align*}
\end{thm}

\begin{proof}
In Table \ref{table: H2Knn as SnxS2-reps}, the isotypical component of $H_2(\uconf_2(K_{n,n}))$ indexed by $\lambda$ is $\lambda\otimes (2)$. We deduce $H_1$ from the Euler characteristic and that $H_0=\Q$.
\end{proof}

\paragraph{\textit{Combinatorial interpretations}} We give a full analysis of $H_2(\conf_2(K_{4,4}))$ as an example. For general $K_{n,n}$, we describe the submodule generated by toric classes.

\begin{example}
We know that \[H_2(\conf_2(K_{4,4}))\cong (1,1,1,1)+3(2,1,1)+2(2,2)+3(3,1)+2(4).\] If $G\subseteq K_{4,4}$ is a subgraph, then the inclusion map induces a map $H_2(\conf_2(G))\to H_2(\conf_2(K_{4,4}))$. The $S_4$-orbit of the image generates a subrepresentation. Next, we give a list of subgraphs of $K_{4,4}$ that induces all subrepresentations.
\begin{enumerate}
    \item The disjoint union of 4-cycles $(1\bar12\bar2)\sqcup(3\bar34\bar4)\hookrightarrow K_{4,4}$ induces a toric submodule isomorphic to $(2,2)+(3,1)+(4)$. Similarly, $(1\bar23\bar4)\sqcup(\bar12\bar34)\hookrightarrow K_{4,4}$ induces an isomorphic submodule, and their intersection is trivial.
    \item The identity map $K_{3,3}\hookrightarrow K_{4,4}$ induces a submodule isomorphic to $\mathbf{(3,1)} + (4)$. The boldfaced part intersects the union of previous modules trivially.
    \item The map $K_{3,3}\hookrightarrow K_{4,4}$ that is identity on $\{1,2,3\}$ and sending $\bar1\mapsto \bar2,\bar2\mapsto \bar3,\bar3\mapsto \bar4$ induces a submodule isomorphic to $\mathbf{(2, 1, 1)} + (2, 2) + (3, 1) + (4)$.
    \item The inclusion $K_{3,4}\hookrightarrow K_{4,4}$ induces a submodule isomorphic to $\mathbf{(2, 1, 1)} + (2,1,1)+(2, 2) + 2(3, 1) + (4)$.
    \item Finally, the graph $K_{4,4}\setminus\{\{4,\bar3\},\{4,\bar4\}\}\hookrightarrow K_{4,4}$ induces the full representation $H_2(\conf_2(K_{4,4}))$.
\end{enumerate}
In particular, this means no $H_2$-class of $\conf_2(K_{4,4})$ uses all its edges.
\end{example}

As we saw in the example, there are two toric submodules. Denote them by $M_1^n$ and $M_2^n$.

\begin{prop}\label{prop: triv in Knn}
The submodule $M_i^n$ is isomorphic to $M((2,2)\oplus(3,1)\oplus(4))$ for $n\geq 4$. Furthermore, $M_1^n\cap M_2^n = \{0\}$ for $n\leq 7$ and $M_1^n\cap M_2^n \cong \Sp_n(4)$ for $n\geq 8$.
\end{prop}

\begin{proof}
We know that $M_i^4$ is isomorphic to $(2,2)\oplus(3,1)\oplus(4)$. To obtain the expression for $M_i^n$, we notice that for $n>4$, the subgroup $S_{\{5,\dots,n\}}$ acts trivially on $M_i^4\subset M_i^n$. We further notice that $M((2,2)\oplus(3,1)\oplus(4))$ is multiplicity stable at $n=8$. Therefore, we check their intersections explicitly using software until stability begins.
\end{proof}

\begin{rmk}
These two toric submodules do not account for all toric classes in $H_2$. For example, for $K_{5,5}$, a toric class is given by $(1\bar12\bar23\bar3)\sqcup(4\bar45\bar5)$. The submodule $N$ generated by this toric class and its $S_5$-images is different from $M_1^5\sqcup M_2^5$: 
\begin{align*}
    N &\cong (2, 1, 1, 1) + (3, 1, 1) + (3, 2) + (4, 1) + (5)\\
    M_{1}^5\cap N &\cong (3, 2) + (4, 1) + (5)\\
    M_2^5\cap N &= 0
\end{align*} The fact that $M_{1}^5\cap N$ is nontrivial can be seen from that the cycles $(1\bar12\bar2),(2\bar23\bar3),(1\bar13\bar3)$ which appear in $M_1^5$ and the cycles $(1\bar12\bar23\bar3),(1\bar13\bar32\bar2)$ which appear in $M_2^5$ use the exact same edges counting multiplicity.
\end{rmk}

\subsection{Crown graphs}

\paragraph{\textit{Computing the bounds}}
Let $W_n$ be the crown graph on $2n$ vertices. The graph $W_n$ has vertex set $\{1,\dots,n\}\cup \{\Bar{1},\dots,\Bar{n}\}$ and edge set $\{\{i,\Bar{j}\}: i\neq j\}$. We orient edges from $i\to\bar j$. The graph $W_n$ is isomorphic to a complete bipartite graph $K_{n,n}$ with a perfect matching removed. Observe that, essentially because of this description, the chain modules of $D_2(W_n)$ form a submodule of the chain modules of the complete bipartite graph. In fact, it is not hard to see that this submodule respects all of the extra FI$\sharp$-structure described in the prior section. Our stable range bounds can now be deduced from this as follows.

\begin{thm} \label{crownBound}
Let $W_n$ denote the crown graph on $2n$ vertices. Then the FI-modules $H_1(\conf_2(W_n))$ and $H_2(\conf_2(W_n))$ are multiplicity stable no later than $n = 7$ and $n = 8$, respectively. The same holds for the unordered configuration spaces.
\end{thm}

\paragraph{\textit{$S_n$-equivariant chain modules}} 
The chain modules of $D_2(W_n)$ and $D^U_2(W_n)$ are submodules of those of $K_{n,n}$ whose generators do not use an edge of the form ${i\bar i}$. We summarize the representation structures below.

\begin{lem}
    The representations afforded by the chain modules of $D_2(W_n)$ are
    \begin{align*}
        C_2(D_2(W_n)) &= M(2)\oplus 2M(3) \oplus M(4),\quad n\geq 4\\
        C_1(D_2(W_n)) &= 4M(2)\oplus 4M(3),\quad n\geq 3\\
        C_0(D_2(W_n)) &= 2M(1)\oplus 4M(2),\quad n\geq 2
    \end{align*}
\end{lem}
\begin{lem}
The representations afforded by the chain modules of $D^U_2(W_n)$ are
\begin{align*}
C_2(D^U_2(W_n)) &= M((1,1))\oplus M(3) \oplus M(2(2,1,1)\oplus 2(3,1)),\quad n\geq 4\\
C_1(D^U_2(W_n)) &= 2M(2)\oplus 2M(3),\quad n\geq 3\\
C_0(D^U_2(W_n)) &= M(1)\oplus 2M((2))\oplus M(2),\quad n\geq 2
\end{align*}
\end{lem}

\paragraph{\textit{Stable homology}}
\begin{thm}\label{thm: ordered crown}
We have the following isomorphisms of $S_n$-representations
\begin{align*}
H_2(\conf_2(W_n)) \cong &\mathrm{Sp}_n(1,1,1,1)\oplus 3\mathrm{Sp}_n(2,1,1) \oplus 2\mathrm{Sp}_n(2,2) \oplus 3\mathrm{Sp}_n(3,1) \oplus \mathrm{Sp}_n(4) \\
&\qquad \oplus 2\mathrm{Sp}_n(1,1,1) \oplus 4\mathrm{Sp}_n(2,1) \oplus 2\mathrm{Sp}_n(3) \oplus 3\mathrm{Sp}_n(1,1) \oplus 3\mathrm{Sp}_n(2) \\
&\qquad \oplus 2\mathrm{Sp}_n(1) \oplus \mathrm{Sp}_n(\emptyset),\quad n\geq 8\\
H_1(\conf_2(W_n))\cong & 2\mathrm{Sp}_n(1, 1) \oplus 2\mathrm{Sp}_n(2),\quad n\geq 5
\end{align*}
\end{thm}
\begin{proof}
    The proof is analogous to the case of $K_{n,n}$.
\end{proof}
\begin{table}[th]
\caption{$S_n\times S_2$-representations afforded by the rational homology of the ordered configuration space of two particles on the crown graph $W_{n}$ for $3\leq n\leq 8$.}
\label{table: H2Crown with SnxS2}
\begin{tabular}{@{}cp{0.80\textwidth}c@{}} \toprule
$n$ & $H_2(\conf_2(W_{n});\Q)$ & Time (s) \\
\midrule
3  & $0$ & $<1$ \\\midrule
4  & $(2,2)\otimes(1,1)+(4)\otimes(1,1)+(3,1)\otimes(2)$ & $<1$ \\\midrule
5  & $(1, 1, 1, 1, 1)\otimes (1, 1)+(2, 1, 1, 1)\otimes (1, 1)+(2, 2, 1)\otimes (1, 1)+(3, 1, 1)\otimes (1, 1)+(3, 2)\otimes (1, 1)+(4, 1)\otimes (1, 1)+(5)\otimes (1, 1)+(2, 1, 1, 1)\otimes (2)+2(3, 1, 1)\otimes (2)+(3, 2)\otimes (2)+(4, 1)\otimes (2)$ & 2 \\\midrule
6  & $(2, 1, 1, 1, 1) \otimes (1, 1)+
(2, 2, 1, 1) \otimes (1, 1)+
2(2, 2, 2) \otimes (1, 1)+
(3, 1, 1, 1) \otimes (1, 1)+
2(3, 2, 1) \otimes (1, 1)+
(4, 1, 1) \otimes (1, 1)+
2(4, 2) \otimes (1, 1)+
(5, 1) \otimes (1, 1)+
(6) \otimes (1, 1)+
2(2, 2, 1, 1) \otimes (2)+
(3, 1, 1, 1) \otimes (2)+
2(3, 2, 1) \otimes (2)+
(3, 3) \otimes (2)+
2(4, 1, 1) \otimes (2)+
(4, 2) \otimes (2)+
(5, 1) \otimes (2)$ & 134  \\\midrule
7  & $(3, 1, 1, 1, 1)\otimes (1, 1) +
(3, 2, 1, 1)\otimes (1, 1) +
2(3, 2, 2)\otimes (1, 1) +
(3, 3, 1)\otimes (1, 1) +
(4, 1, 1, 1)\otimes (1, 1) +
2(4, 2, 1)\otimes (1, 1) +
(4, 3)\otimes (1, 1) +
(5, 1, 1)\otimes (1, 1) +
2(5, 2)\otimes (1, 1) +
(6, 1)\otimes (1, 1) +
(7)\otimes (1, 1) +
2(3, 2, 1, 1)\otimes (2) +
2(3, 3, 1)\otimes (2) +
(4, 1, 1, 1)\otimes (2) +
2(4, 2, 1)\otimes (2) +
(4, 3)\otimes (2) +
2(5, 1, 1)\otimes (2) +
(5, 2)\otimes (2) +
(6, 1)\otimes (2)$  & 8963  \\\midrule
8 & $(4, 1, 1, 1, 1) \otimes (1, 1) + 
(4, 2, 1, 1) \otimes (1, 1) + 
2(4, 2, 2) \otimes (1, 1) + 
(4, 3, 1) \otimes (1, 1) + 
(4, 4) \otimes (1, 1) + 
(5, 1, 1, 1) \otimes (1, 1) + 
2(5, 2, 1) \otimes (1, 1) + 
(5, 3) \otimes (1, 1) + 
(6, 1, 1) \otimes (1, 1) + 
2(6, 2) \otimes (1, 1) + 
(7, 1) \otimes (1, 1) + 
(8) \otimes (1, 1) + 
2(4, 2, 1, 1) \otimes (2) + 
2(4, 3, 1) \otimes (2) + 
(5, 1, 1, 1) \otimes (2) + 
2(5, 2, 1) \otimes (2) + 
(5, 3) \otimes (2) + 
2(6, 1, 1) \otimes (2) + 
(6, 2) \otimes (2) + 
(7, 1) \otimes (2)$  & 203585  \\
\bottomrule
\end{tabular}

\end{table}

\begin{thm}
We have the following isomorphisms of $S_n$-representations
\begin{align*}
H_2(\uconf_2(W_n)) \cong & 2\mathrm{Sp}_n(2,1,1) \oplus 2\mathrm{Sp}_n(3,1) \oplus \mathrm{Sp}_n(1,1,1) \oplus 2\mathrm{Sp}_n(2,1)\oplus \mathrm{Sp}_n(3) \\
& \qquad \oplus 2\mathrm{Sp}_n(1,1) \oplus \mathrm{Sp}_n(2)\oplus \mathrm{Sp}_n(1),\quad n\geq 7.
\\
H_1(\uconf_2(W_n))\cong & \mathrm{Sp}_n(1, 1) \oplus \mathrm{Sp}_n(2),\quad n\geq 5
\end{align*}
\end{thm}

\paragraph{\textit{Combinatorial interpretations}} Notice that the submodule $M_2^n$ in the previous subsection is defined on $W_n\subset K_{n,n}$. Therefore, we obtain the following proposition.

\begin{prop}
The submodule generated by the toric class $(1\bar23\bar4)\sqcup (\bar12\bar34)$ and its $S_n$ images is isomorphic to $M((2,2))\oplus M((3,1))\oplus M((4)))$ for $n\geq 4$.
\end{prop}

\subsection{Complete tripartite graphs $K_{n,n,1}$}

In this section we analyze the complete tripartite graph $K_{n,n,1}$. We focus on this family specifically because of its connection to intrinsic linkage. Specifically, the class $K_{3,3,1}$ is a member of the Petersen family.

\paragraph{\textit{Computing the bounds}}
Let $K_{n,n,1}$ be the complete tripartite graph on $2n+1$ vertices. It has vertex set $\{1,\dots,n\}\cup \{\bar{1},\dots,\bar{n}\}\cup\{0\}$ and edge set $\{\{i,\bar{j}\}: i\neq j\}\cup \{\{0,v\}:v=1,\dots,n,\bar1,\dots,\bar n\}$. The edges are oriented as $i\to\bar j$ and $0\to v$. The symmetric group $S_n$ acts on $K_{n,n,1}$ by permuting $i$ and $\bar i$ simultaneously. Since the $S_n$-action on $K_{n,n,1}$ is completely determined by the action on $K_{n,n}\subset K_{n,n,1}$, the stable range bounds on the homology of configuration spaces of $K_{n,n,1}$ agree with those for $K_{n,n}$.

\begin{thm} \label{thm: tripartiteBounds}
Let $K_{n,n,1}$ denote the complete tripartite graph on $2n+1$ vertices. Then the FI-modules $H_1(\conf_2(K_{n,n,1});\Q)$ and $H_2(\conf_2(K_{n,n,1});\Q)$ are multiplicity stable no later than $n = 7$ and $n = 8$, respectively. The same holds for the unordered configuration spaces.
\end{thm}


\paragraph{\textit{$S_n$-equivariant chain modules}}
\begin{lem}
The representations afforded by chain modules of $D_2(K_{n,n,1})$ are
\begin{align*}
C_2(D_2(K_{n,n,1})) &\cong 10M(2) \oplus 8M(3) \oplus M(4),\quad n\geq 4\\
C_1(D_2(K_{n,n,1})) &\cong 6M(1) \oplus 18M(2) \oplus 4M(3),\quad n\geq 3
\\
C_0(D_2(K_{n,n,1})) &\cong 6M(1) \oplus 4M(2),\quad n\geq 2 
\end{align*}
\end{lem}

\begin{proof}
Each chain module is the direct sum of the corresponding chain module of $K_{n,n}$ and additional combinatorial types of cells. For $C_2$, the extra combinatorial types are \[i\bar i\times 0j,\; i\bar i\times 0\bar j\;, i\bar j \times 0\bar i,\;,i\bar j \times 0j,\;,i\bar j \times 0k,\; i\bar j \times 0\bar k,\] and their $S_2$-images by reversing the ordering of the two edges. Each combinatorial type contributes a direct summand of $M(a)$ where $a$ is the number of letters appearing.

For $C_1$, the extra combinatorial types are \[i\bar i\times 0,i\bar j \times 0,0i\times \bar i, 0i\times j,0i\times \bar j,0\bar i\times i,0\bar i\times j,0\bar i\times \bar j,\] and their $S_2$-images. Each contributes a direct summand of $M(a)$ as described above.

For $C_0$, the extra combinatorial types are \[i\times0,\bar i\times 0,\] and their $S_2$-images. Each contributes a $M(1)$.
\end{proof}

\begin{lem}
The representations afforded by chain modules of $D^U_2(K_{n,n,1})$ are
\begin{align*}
C_2(D^U_2(K_{n,n,1})) &\cong 4M(2)\oplus 2M((1,1))\oplus 4M(3)\oplus M(2(2,1,1)\oplus 2(3,1)),\quad n\geq 4\\
C_1(D^U_2(K_{n,n,1})) &\cong 3M(1)\oplus 9M(2)\oplus 2M(3),\quad n\geq 3\\
C_0(D^U_2(K_{n,n,1})) &\cong 3M(1)\oplus 2M((2)) \oplus M(2),\quad n\geq 2
\end{align*}
\end{lem}

\begin{proof}
    The chain modules of $D^U_2(K_{n,n,1})$ are direct sums of the corresponding chain modules of $D^U_2(K_{n,n})$ and groups generated by additional combinatorial types. We list the extra types and the representations they contribute. For $C_2$, the extra types are \[0i\times j\bar i,0\bar i\times i\bar j,0i\times j\bar j,0\bar i\times j\bar j,0i\times j\bar k,0\bar i \times j\bar k.\] Since none of them has a stabilizer, they each generate $M(a)$ where $a$ is number of letters that appear.

    For $C_1$, the extra types are \[0i\times \bar i,0\bar i\times i,i\bar i\times0,0i\times j,0i\times \bar j,0\bar i\times j,0\bar i\times\bar j,i\bar j\times 0.\] They each contribute $M(a)$ as described above.

    For $C_0$, the extra types are \[0\times i,0\times \bar i.\] They each contribute $M(1)$.
\end{proof}

\paragraph{\textit{Stable homology}}
\begin{thm}\label{thm: ordered tripartite}
    There are isomorphisms of $S_n$-representations
    \begin{align*}
        H_2(\conf_2(K_{n,n,1})) &\cong \Sp_n(1, 1, 1, 1) \oplus 3\Sp_n(2, 1, 1) \oplus 2\Sp_n(2, 2) \oplus 3\Sp_n(3, 1) \oplus \Sp_n(4) \\
        &\qquad \oplus 8\Sp_n(1, 1, 1) \oplus 16\Sp_n(2, 1) \oplus 8\Sp_n(3) \oplus 16\Sp_n(1, 1) \oplus 16\Sp_n(2) \\
        &\qquad\oplus 14\Sp_n(1) \oplus 4\Sp_n(\emptyset),\quad n\geq 8\\
        H_1(\conf_2(K_{n,n,1})) &\cong 2\Sp_n(1, 1) \oplus 2\Sp_n(2) \oplus 6\Sp_n(1) \oplus 4\Sp_n(\emptyset), \quad n\geq 4
    \end{align*}
\end{thm}

\begin{proof}
    We use Table \ref{table: H2tripartite} and the Euler characteristic of $D_2(K_{n,n,1})$.
\end{proof}

\begin{thm}
We have the following isomorphisms of $S_n$-representations
\begin{align*}
    H_2(\uconf_2(K_{n,n,1})) &\cong 2\Sp_n(2, 1, 1) \oplus 2\Sp_n(3, 1) \oplus 4\Sp_n(1, 1, 1) \oplus 8\Sp_n(2, 1)\oplus 4\Sp_n(3)  \\
    &\qquad \oplus 9\Sp_n(1, 1) \oplus 7\Sp_n(2) \oplus 7\Sp_n(1) \oplus \Sp_n(\emptyset),\quad n\geq 7\\
    H_1(\uconf_2(K_{n,n,1})) &\cong \Sp_n(1, 1) \oplus \Sp_n(2) \oplus 3\Sp_n(1) \oplus 2\Sp_n(\emptyset), \quad n\geq 4
\end{align*}
\end{thm}

\begin{proof}
    We use Table \ref{table: H2tripartite unordered} and the Euler characteristic of $D^U_2(K_{n,n,1})$.
\end{proof}




\begin{table}[th]
\caption{$S_n$-representations afforded by the rational homology of the ordered configuration space of two particles on the complete tripartite graph $K_{n,n,1}$ for $3\leq n\leq 8$.}
\label{table: H2tripartite}
\begin{tabular}{@{}cp{0.80\textwidth}c@{}} \toprule
$n$ & $H_2(\conf_2(K_{n,n,1});\Q)$ & Time (s) \\
\midrule
3  & $2(1, 1, 1) + 6(2, 1) + 3(3)$ & $<1$ \\\midrule
4  & $5(1, 1, 1, 1) + 13(2, 1, 1) + 8(2, 2) + 13(3, 1) + 4(4)$ & 4 \\\midrule
5  & $(1, 1, 1, 1, 1) + 8(2, 1, 1, 1) + 13(2, 2, 1) + 16(3, 1, 1) + 15(3, 2) + 14(4, 1) + 4(5)$ & 185 \\\midrule
6  & $(2, 1, 1, 1, 1) + 3(2, 2, 1, 1) + 2(2, 2, 2) + 8(3, 1, 1, 1) + 16(3, 2, 1) + 7(3, 3) + 16(4, 1, 1) + 16(4, 2) + 14(5, 1) + 4(6)$ & 11343  \\\midrule
7  & $(3, 1, 1, 1, 1) + 3(3, 2, 1, 1) + 2(3, 2, 2) + 3(3, 3, 1) + 8(4, 1, 1, 1) + 16(4, 2, 1) + 8(4, 3) + 16(5, 1, 1) + 16(5, 2) + 14(6, 1) + 4(7)$ & 206791  \\\midrule
8 & $(4, 1, 1, 1, 1) + 3(4, 2, 1, 1) + 2(4, 2, 2) + 3(4, 3, 1) + (4, 4) + 8(5, 1, 1, 1) + 16(5, 2, 1) + 8(5, 3) + 16(6, 1, 1) + 16(6, 2) + 14(7, 1) + 4(8)$ & Use $H_1$ \\
\bottomrule
\end{tabular}

\end{table}

\begin{table}[th]
\caption{$S_n$-representations afforded by the rational homology of the unordered configuration space of two particles on the complete tripartite graph $K_{n,n,1}$ for $3\leq n\leq 8$.}
\label{table: H2tripartite unordered}
\begin{tabular}{@{}cp{0.80\textwidth}c@{}} \toprule
$n$ & $H_2(\uconf_2(K_{n,n,1});\Q)$ & Time (s) \\
\midrule
3  & $(1, 1, 1) + 3(2, 1) + (3)$ & $<1$ \\\midrule
4  & $2(1, 1, 1, 1) + 7(2, 1, 1) + 3(2, 2) + 7(3, 1) + (4)$ & $<1$ \\\midrule
5  & $4(2, 1, 1, 1) + 6(2, 2, 1) + 9(3, 1, 1) + 7(3, 2) + 7(4, 1) + (5)$ & 17 \\\midrule
6  & $2(2, 2, 1, 1) + 4(3, 1, 1, 1) + 8(3, 2, 1) + 4(3, 3) + 9(4, 1, 1) + 7(4, 2) + 7(5, 1) + (6)$ &
873  \\\midrule
7  & $2(3, 2, 1, 1) + 2(3, 3, 1) + 4(4, 1, 1, 1) + 8(4, 2, 1) + 4(4, 3) + 9(5, 1, 1) + 7(5, 2) + 7(6, 1) + (7)$  & 14877  \\\midrule
8 &  $2(4, 2, 1, 1) + 2(4, 3, 1) + 4(5, 1, 1, 1) + 8(5, 2, 1) + 4(5, 3) + 9(6, 1, 1) + 7(6, 2) + 7(7, 1) + (8)$ & 181161  \\
\bottomrule
\end{tabular}

\end{table}



\paragraph{\textit{Combinatorial interpretations}}
We give a partial analysis of the generators of the trivial representatios in $H_2(\conf_2(K_{3,3,1}))$.

\begin{prop}
    At least two out of three trivial representations in $H_2(\conf_2(K_{3,3,1}))$ are generated by toric classes. In particular, the non-toric class induced by $K_{3,3}\hookrightarrow K_{3,3,1}$ is a linear combination of two toric classes.
\end{prop}

\begin{proof}
    We identify two pairs of disjoint cycles in $K_{3,3,1}$: $(01\bar1) \sqcup (2\bar23\bar3)$ and $(03\bar1)\sqcup(2\bar21\bar3)$. Each pair defines a toric class in $H_2(\conf_2(K_{3,3,1}))$ and their $S_3$-average each generates a copy of trivial representation. Write $b_1$ for the generator coming from $(01\bar1) \sqcup (2\bar23\bar3)$ and $b_2$ the other pair, we compute explicitly that the class induced by $K_{3,3}\hookrightarrow K_{3,3,1}$ is $\frac{1}{2}b_1-b_2$.
\end{proof}

\begin{remark}
    As $K_{3,3,1}$ contains no other pair of disjoint cycles up to $S_3$-equivalence, we conjecture that the generator of the third copy of trivial representation is non-toric.
\end{remark}

\subsection{Kneser graphs $K(n,2)$}

\paragraph{\textit{Computing the bounds}} The Kneser graph $K(n,2)$ has vertex set $\binom{[n]}{2}$, i.e., subsets of size two of $[n]$, and two vertices are adjacent if and only if they are disjoint as subsets of $[n]$. An edge is oriented $\{i,j\}\to \{k,l\}$ whenever $\min\{i,j\}<\min\{k,l\}$. The action from $S_n$ is induced by its action on the underlying set $[n]$. As before, the chain modules of the discretized configuration space $D_2$ are free. As we need at most eight indices to define a 2-cell, the stable degree $s(C_2)=8$. Similarly, $s(C_1)=6$ and $s(C_0)=4$. Using Theorem \ref{relationBound}, we deduce the following bounds.

\begin{thm}
    The FI-modules $H_2(\conf_2(K(n,2)))$ and $H_1(\conf_2(K(n,2)))$ are multiplicity stable no later than 16 and 14, respectively. The same holds for $\uconf_2(K(n,2))$.
\end{thm}

\paragraph{\textit{Stable homology}} As the stable range is beyond our computational power, we record partial results for small $n$ for the ordered case in the following table.

\begin{table}[th]
\caption{$S_n$-representations afforded by the rational homology of the ordered configuration space of two particles on the Kneser graph $K(n,2)$ for $3\leq n\leq 8$.}
\label{table: H2Kneser}
\begin{tabular}{@{}cp{0.80\textwidth}c@{}} \toprule
$n$ & $H_2(\conf_2(K(n,2));\Q)$ & Time (s) \\
\midrule
4  & 0 & $<1$ \\\midrule
5  & $(1, 1, 1, 1, 1) + (2, 2, 1) + (3, 2)$ & $<1$ \\\midrule
6  & $(1, 1, 1, 1, 1, 1) + 6(2, 1, 1, 1, 1) + 7(2, 2, 1, 1) + 6(2, 2, 2) + 10(3, 1, 1, 1) + 14(3, 2, 1) + 3(3, 3) + 8(4, 1, 1) + 8(4, 2) + 3
(5, 1) + (6)$ & 5868  \\
\bottomrule
\end{tabular}

\end{table}

\section{The case of three particles}

\paragraph{\textit{Computing the bounds}}

The cells in $D_3(K_n)$ are defined over the graph $K_n'$, which is obtained from $K_n$ by subdividing each edge once. We label the vertices of $K_n$ by $[n]$ and the vertices of $K_n'$ by $[n]\cup\{ij:i\neq j\}$ where $ij$ is the new vertex subdividing the edge $\{i,j\}$. An edge in $K_n'$ is oriented as $ij \to i$. The chain modules $C_i$ are always free and the generating degrees are $s(C_3)=6, s(C_2) = 6, s(C_1) = 6, s(C_0) = 6$. Using Theorem \ref{relationBound}, we obtain the following.

\begin{thm} \label{completeBound}
Let $K_n$ denote the complete graph on $n$ vertices. Then the FI-module $H_i(\conf_3(K_n);\Q)$ is multiplicity stable no later than $n = 12$, for $i=0,1,2,3$. The same holds for $\uconf_3(K_n)$.
\end{thm}

\paragraph{\textit{$S_n$-equivariant chain modules}}

\begin{lem}
The cellular chain modules of $D_3(K_n)$ have the following descriptions as $S_n$-representations, for $n\geq 6$.
\begin{align*}
C_3 &\cong 2M(3) \oplus 16M(4)\oplus 9M(5)\oplus M(6)\\
C_2 &\cong 18M(3) \oplus 54M(4) \oplus 15M(5) \oplus 9M(\ind_{S_3\times S_2}^{S_5} \reg_3\otimes \Q)\oplus 3M(\ind_{S_4\times S_2}^{S_6}\reg_4\otimes \Q)\\
C_1 &\cong 36M(3)\oplus 45M(4)\oplus 6M(\ind_{S_2\times S_2}^{S_4}\reg_2\otimes \Q)\oplus 3M(5)\oplus 18M(\ind_{S_3\times S_2}^{S_5}\reg_3\otimes \Q)\\
&\qquad \oplus 3M(\ind_{S_2^{\times 3}}^{S_6} \reg_2 \otimes \Q \otimes \Q)\\
C_0 &\cong 3M(2) \oplus 17M(3) \oplus 7M(4) \oplus 9M(\ind_{S_2\times S_2}^{S_4}\reg_2\otimes \Q) \oplus 3M(\ind_{S_2\times S_2}^{S_5} \Q\otimes \Q) \\
&\qquad \oplus 3M(\ind_{S_3\times S_2}^{S_5}\reg_3\otimes \Q) \oplus M(\ind_{S_2^{\times 3}}^{S_6} \Q^{\otimes 3})
\end{align*}

\end{lem}

\begin{proof}
Each chain module of $D_3(K_3)$ is generated by cells which are partitioned into combinatorial types. We enumerate all combinatorial types in each chain module, see Appendix \ref{appendix: combo types of D3Kn}. Note that the Appendix only records the combinatorial types for the unordered configuration space; for the ordered case, one also needs to take into consideration different ordering of the edges and vertices. We omit the description for simplicity. The full list can be found in the code. Then for each combinatorial type, we compute the representation afforded on the submodule it generates.
\end{proof}

\paragraph{\textit{Stable homology}}
We briefly comment on our algorithm for computing homology representations in this case. The chain modules involved are large; for example, $C_2(\conf_3(K_6))$ has dimension 36720 and $C_1(\conf_3(K_6))$ has dimension 30780. To avoid doing linear algebra on these large spaces, we take advantage of the representation structures. We rely on the following lemma, which is a consequence of Schur's Lemma. Recall that $\Sp(\lambda)$ is the Specht module indexed by the partition $\lambda$.
\begin{lem}
    Suppose $f:V\to V'$ is a morphism of $S_n$-representations. Suppose further that $V=V_{\lambda_1}\oplus \cdots\oplus V_{\lambda_n}$ where each $V_{\lambda_i}\cong m_{i}\Sp(\lambda_i)$. Then the multiplicity of $\Sp(\lambda_i)$ in $\ker f$ as an $S_n$-representation is the nullity of $f|_{V_{\lambda_i}}$ divided by $\dim(\Sp(\lambda_i))$.
\end{lem}

However, in practice, when $V$ has a large dimension, it is often expensive to compute a basis for $V_{\lambda_i}$. We simplify the computations by observing that each chain module $C$ is a direct sum of submodules $C_1
\oplus \cdots \oplus C_k$, each given by a combinatorial type of cells. Therefore, we compute $f|_{C_j|_{V_{\lambda_i}}}$ and then compute the dimension of the span of $\cup_j f|_{C_j|_{V_{\lambda_i}}}$.

We were able to compute homology representations for $\conf_3(K_n)$ up to $n=7$; the results are recorded in Table \ref{table: H3Kn}. This allows us gather information on the stable multiplicity of irreducible subrepresentations indexed by certain partitions.

\begin{thm}\label{thm: three particles on complete}
For $n\geq 6$, the multiplicity of the trivial representation $\mathrm{Sp}_n(\emptyset)$ in $H_i(\conf_3(K_n))$ is zero for $i=3,1,0$ and in $H_2(\conf_3(K_n))$ is two. For $n\geq 7$, the multiplicity of $\mathrm{Sp}_n(1)$ in $H_i(\conf_3(K_n))$ is zero for $i=3,1,0$ and in $H_2(\conf_3(K_n))$ is six.
\end{thm}

\begin{proof}
The homology group $H_i(\conf_3(K_n))$ is computed by the following part of the cellular chain complex for $D_3(K_n)$: \[C_{i+1}(D_3)\to C_i(D_3) \to C_{i-1}(D_3).\] We know that $s(C_i(D_3))=6$ for $i\leq 3$ and it is zero for $i>3$. By Theorem \ref{relationBound}, the multiplicity of the trivial representation in $H_i(\conf_3(K_n))$ stabilizes at $n=6+0=6$. The multiplicity of $\mathrm{Sp}_n(1)$ stabilizes at $n=6+1=7$.
\end{proof}


\begin{table}[th]
\caption{$S_n$-representations afforded by the degree-3 rational homology of the ordered configuration space of two particles on the complete graph $K_n$ for $4\leq n\leq 7$.}
\label{table: H3Kn}
\begin{tabular}{@{}ccp{0.80\textwidth}c@{}} \toprule
$n$ & $i$ & $H_i(\conf_3(K_{n});\Q)$ & Time (s) \\
\midrule
4  & 3 & $0$ & 2 \\
 & 2 & $2(4) + 3(3,1)$ & 1404\\
 & 1 & $3(2,1,1) + 3(1,1,1,1)$ & \\
 & 0 & $(4)$ & \\\midrule
5  & 3 & $0$ & 104 \\
& 2 & $3(1,1,1,1,1) + 6(2,1,1,1) + 6(2,2,1) + 9(3,1,1) + 6(3,2) + 6(4,1) + 2(5)$ & 7034 \\
& 1 & $3(3,1,1)$ &  \\
& 0 & $(5)$ &  \\\midrule
6  & 3 & $0$ & 1555  \\
& 2 & $2(1, 1, 1, 1, 1, 1) + 7(2, 1, 1, 1, 1) + 12(2, 2, 1, 1) + 8(2, 2, 2) + 8(3, 1, 1, 1) + 16(3, 2, 1) + 5(3, 3) + 9(4, 1, 1) + 9(4, 2) + 6(5, 1) + 2(6)$  & 32062 \\
& 1 & $3(4,1,1)$  &  \\
& 0 & $(6)$ &  \\\midrule
7  &3 & $(1^7)$  & 17157  \\
\bottomrule
\end{tabular}
\end{table}


The data on the configuration spaces of one, two, and three particles on the complete graph $K_n$ led us to the following conjecture.
\begin{conjecture}
    For $n\geq 5$, there is an isomorphism of $S_n$-representations \[H_1(\conf_k(K_n))\cong k\,\Sp_n(1,1).\]
\end{conjecture}

Computations were done in SageMath\cite{sagemath} on the second author's personal laptop and the computer server at the Department of Mathematics and Statistics at UiT, with the
following system specifics, respectively: 8-core Apple M3 CPU on Darwin 25.4.0 (2026-03-19) arm64; 2x24-core Intel Xeon Platinum 8168 CPU (2.70 GHz) on Ubuntu GNU/Linux 6.8.0-101-generic (2026-02-09) x86\_64.


\bibliographystyle{alpha}
\bibliography{ref}

\appendix
\section{Combinatorial types of cells in $D^3(K_n)$}
\label{appendix: combo types of D3Kn}
\begin{table}[th]
\caption{Combinatorial types of 3-dimensional cells in $D^3(K_n)$.}
\label{table: 3 cells in D^3(K_n)}
\begin{tabular}{@{}>{\centering\arraybackslash}m{2cm}p{0.8\textwidth}@{}} \toprule
Generated in degree $n$ & Combinatorial types of 3-dimensional cells in $D^3(K_n)$\\
\midrule
3 & \raisebox{-.5\height}{\begin{tikzpicture}
\begin{scope}[scale=0.8]
\coordinate (i) at (1,1);
\coordinate (j) at (0,0);
\coordinate (k) at (2,0);
\draw[thick] (i) -- (j) -- (k) -- cycle;
    
\filldraw[fill=black, draw=black] (j) circle [radius=2pt];
\filldraw[fill=black, draw=black] (k) circle [radius=2pt];
\filldraw[fill=black, draw=black] (i) circle [radius=2pt];
\filldraw[fill=black, draw=black] (1,0) circle [radius=2pt];
\filldraw[fill=black, draw=black] (0.5,0.5) circle [radius=2pt];
\filldraw[fill=black, draw=black] (1.5,0.5) circle [radius=2pt];

\node[above] at (1,1) {$i$};
\node[left] at (0,0) {$j$};
\node[right] at (2,0) {$k$};

\draw[line width=2pt] (i) -- (0.5,0.5);
\draw[line width=2pt] (j) -- (1,0);
\draw[line width=2pt] (k) -- (1.5,0.5);
\end{scope}
\end{tikzpicture}} \\
\bottomrule
\end{tabular}

\end{table}


\begin{table}[th]
\caption{Combinatorial types of 2-dimensional cells in $D^3(K_n)$.}
\label{table: 2 cells in D^3(K_n)}
\begin{tabular}{@{}>{\centering\arraybackslash}m{2cm}p{0.8\textwidth}@{}} \toprule
Generated in degree $n$ & Combinatorial types of 2-dimensional cells in $D^3(K_n)$\\
\bottomrule
\end{tabular}

\end{table}

\begin{table}[th]
\caption{Combinatorial types of 1-dimensional cells in $D^3(K_n)$.}
\label{table: 1 cells in D^3(K_n)}
\begin{tabular}{@{}>{\centering\arraybackslash}m{2cm}p{0.8\textwidth}@{}} \toprule
Generated in degree $n$ & Combinatorial types of 1-dimensional cells in $D^3(K_n)$\\
\bottomrule
\end{tabular}
\end{table}

\begin{table}[th]
\caption{Combinatorial types of 0-dimensional cells in $D^3(K_n)$.}
\label{table: 0 cells in D^3(K_n)}
%
\end{table}

\end{document}